\documentclass[final,leqno,onefignum,onetabnum]{siamltex1213}

\usepackage[utf8x]{inputenc}
\usepackage{amsmath,amssymb}
\usepackage{cite,hyperref,url,float}

\title{Convexity of a stochastic control functional related to importance sampling of It\^{o} diffusions} 

\author{Han Cheng Lie\footnotemark[2]\ \footnotemark[3]
}

\newtheorem{mythm}{Theorem}[section]
\newtheorem{myprop}[mythm]{Proposition}
\newtheorem{mylem}[mythm]{Lemma}
\newtheorem{mycor}[mythm]{Corollary}
\newtheorem{myasmp}[mythm]{Assumption}

\begin{document}
\maketitle

\renewcommand{\thefootnote}{\fnsymbol{footnote}}

\footnotetext[2]{Institut f\"{u}r Mathematik, Freie Universit\"at Berlin, Berlin, Germany (\email{hlie@math.fu-berlin.de}).
}
\footnotetext[3]{This research was funded by a Ph.D. scholarship from the International Max-Planck Research School for Computational Biology and Scientific Computing of the Max Planck-Institut f\"{u}r molekulare Genetik, and by the Deutsche Forschungsgemeinschaft (DFG) through grant CRC 1114.}

\renewcommand{\thefootnote}{\arabic{footnote}}

\slugger{sima}{xxxx}{xx}{x}{x--x}

\begin{abstract}
We consider the problem of rare event importance sampling, where the random variable of interest is a path functional of an It\^{o} diffusion computed up to the first exit from a $d$-dimensional bounded domain. Dupuis and Wang (\textit{Ann. Appl. Probab.}, 15 (2005), pp. 1-38) studied the importance sampling problem by formulating it as a stochastic optimal control problem, where the value function is related to the conditional cumulant generating function of the random variable. In this paper, we show that the sufficient conditions for the value function to be twice-differentiable and $\alpha$-uniformly H\"{o}lder continuous on the closure of the domain are also sufficient conditions for positive definiteness of the second variation of the control functional on the space of differentiable, $\alpha$-uniformly H\"{o}lder continuous $\mathbb{R}^d$-valued feedback controls. We derive an expression for the second variation using Kazamaki's sufficient condition for $L^q$-boundedness of exponential martingales, and using Fredholm theory to prove the finiteness of the moment generating function of the first exit time over any bounded interval containing the origin. The strict convexity result suggests that one may be able to solve the corresponding Hamilton-Jacobi-Bellman boundary value problem in a dimension-robust way, by combining convex optimisation and Monte Carlo methods. We apply the result to analyse a gradient descent algorithm proposed by Hartmann and Sch\"{u}tte (\textit{J. Stat. Mech. Theor. Exp.} (2012), P11004) for efficient rare event simulation.
\end{abstract}



\pagestyle{myheadings}
\thispagestyle{plain}
\markboth{H. C. LIE}{CONVEXITY OF A STOCHASTIC CONTROL FUNCTIONAL}

\section{Introduction}
\label{sec:intro}

The estimation of statistical properties of rare events is a common problem in computational physics. For example, the estimation of thermodynamic free energy differences between two equilibrium states is integral to the study of complex macromolecules in computational biophysics and drug design. Accordingly, many methods for rare event estimation have been proposed, especially in the case of free energy computations. Some methods, e.g. the adaptive biasing force method \cite{darve_adaptivebiasingforce} or hyperdynamics \cite{voter_hyperdynamics,voter_methodaccelerating}, involve driving the dynamical system of interest by applying an external force; in certain cases, this is achieved by changing the (potential) energy landscape on which the system lives, e.g. metadynamics \cite{laio_parrinello_pnas}. Other methods, such as the Passerone-Parrinello method \cite{passerone}, study a variational problem and apply a least action principle, while others focus on improving the simulation of rare trajectories that dominate the statistical property of interest, e.g. the transition path sampling method \cite{tps1998} or the forward flux sampling method \cite{allenfrenkelreintenwolde}. A mathematical treatment of free energy computations is given in \cite{free_energy_computations}.

The motivation of the present article is the analysis of a gradient descent method for solving a stochastic optimal control problem, proposed by Hartmann and Sch\"{u}tte \cite{hartmannschuette_efficient}. This method shares some characteristics with the methods mentioned above: the method improves the collection of rare event statistics by tilting the potential in order to bias the path measure towards trajectories that dominate an exponential average (see, e.g. \cite{jarzynski_pre}), each modification to the potential corresponds to the application of a feedback control force, and a variational principle enters via the duality relationship between free energy and relative entropy (see, e.g. \cite{daipra}). 

In this article, we describe sufficient conditions under which the control functional of the stochastic optimal control problem is strictly convex, by rigorously deriving an expression for the second variation of the control functional, and by showing that a deterministic lower bound on the path functional suffices to guarantee strict positive definiteness of the second variation. Strict convexity implies the well-posedness of the restriction of the stochastic optimal control problem to any finite-dimensional subspace of the set of admissible feedback control functions. In particular, in the limit of infinitely large sample size and infinitesimally small descent step size, it follows that the gradient descent method mentioned above converges to the unique minimiser corresponding to the best approximation in the finite-dimensional subspace. For the functionals and systems considered here, the strict convexity result renders the importance sampling problem amenable to other methods for solving convex optimisation problems that may be more efficient than gradient descent. This is important, because the connections between the optimisation problems considered here and certain elliptic Dirichlet boundary value problems imply that one can obtain numerical methods for solving Hamilton-Jacobi-Bellman boundary value problems that scale well with the spatial dimension.

Our initial reference for the idea of formulating an importance sampling problem as a stochastic optimal control problem is \cite{hartmannschuette_efficient}, which examined the problem of accelerating simulations of rare events by using the connection between risk-sensitive, exit time control and importance sampling \cite{dupuis_mceneaney} and the aforementioned free energy-relative entropy duality relationship \cite{daipra}; the work \cite{dupuis2005} also studied stochastic optimal control formulations of importance sampling problems. The connections between importance sampling, large deviations, and differential games has been studied in \cite{WangDupuis2004}, these ideas have been further developed for rare event simulation of multiscale diffusions \cite{dupuis_wang_spiliopoulos,WeareEVE}. We shall not consider the large deviations aspect of rare event importance sampling. Instead, we shall view importance sampling as an optimisation problem defined over a set of measures; this idea has been applied to Markov decision processes in \cite{borkar1988}, and in the cross-entropy minimisation method introduced by Kullback \cite{Kullback59informationtheory}. Cross-entropy minimisations have been applied to importance sampling for rare event estimation by Rubenstein in the work \cite{Rubinstein199789}, and recently by Asmussen, Kroese, and Rubenstein in \cite{asmussen_kroese_rubinstein}. In molecular dynamics, importance sampling and cross-entropy minimisation have been treated in \cite{hartmann_char}, and Zhang et. al. \cite{sisc_crossentropy}. 

The present work may differ from some of those mentioned above in the following ways. First, although we first consider the importance sampling problem, our primary concern is an analysis of the stochastic optimal control problem via the control functional. Second, our analysis is motivated by the desire to better understand what numerical methods may be applied to solve such problems efficiently in high-dimensional domains. Third, we employ functional analytic tools from the theory for classical solutions to elliptic partial differential equations, as opposed to large deviations techniques. Furthermore, since the approach we consider does not involve minimising the relative entropy alone, it differs from the cross-entropy minimisation methods mentioned earlier. To the best of our knowledge, no work to date has addressed the question of whether the control functional is convex over a suitably chosen function space. An advantage of our functional-analytic approach is that the sufficient conditions we specify are deterministic; such conditions may be easier to verify in practical situations.

\section{Overview}
\label{sec:overview}

Let $d\in\mathbb{N}$ be the dimension of the state space (configuration space) of a stochastic dynamical system whose paths belong to the canonical filtered probability space $(\Theta,\mathcal{F},(\mathcal{F}_t)_{t\ge 0}, P)$, where $\Theta= C([0,\infty);\mathbb{R}^d)$ and $P$ is the standard Wiener measure. Consider an It\^{o} diffusion, i.e. a stochastic dynamical system with continuous paths that satisfies the stochastic differential equation (SDE)
\begin{equation}\label{eq:controlled_sde}
 dX^u_t=\left(u-\nabla V\right)(X^u_t)dt+\sqrt{2\varepsilon}dB_t,\hskip1ex X^u_0=x,
\end{equation}
where $u=(u_t)_{t\ge 0}$ denotes a predictable control process, $V\in C^1(\mathbb{R}^d;\mathbb{R})$ is an (potential) energy function that is bounded from below and that satisfies the growth conditions
\begin{subequations}\label{eq:growth_conditions}
 \begin{align}
 \vert V(x')\vert &\le K_0(1+\vert x'\vert)
 \label{eq:linear_growth_condition_onV}
 \\
\vert \nabla V(x')\vert^2&\le K_1^2(1+\vert x'\vert^2)
 \label{eq:quadratic_growth_condition_on_gradV}
  \end{align}
\end{subequations}
for some constants $K_0,K_1>0$ and for all $x'\in\mathbb{R}^d$, $\varepsilon>0$ is a fixed temperature, and $B$ is the standard Brownian motion with respect to $P$. In the context of computational physics and molecular dynamics, the SDE (\ref{eq:controlled_sde}) is the overdamped Langevin equation that models the evolution of a large molecule in equilibrium with a heat bath composed of many smaller solvent molecules, where $\nabla V$ corresponds to the force field that arises from bonded and non-bonded interactions of the constituent atoms, and the white noise term models the effect of the heat bath on the molecule. The control $u$ models the effect of a controllable external force that acts on the molecule, such as a pressure or electromagnetic field. Given $u$, the corresponding solution to (\ref{eq:controlled_sde}) is a random variable $X^u:\Theta\to\Theta$. Letting $P^x(A):=P(\theta\in A\vert \theta_0=x)$ denote the conditioning of $P$ on the deterministic initial condition $x\in\mathbb{R}^d$, we then write the law of $X^u$, its conditioned version, and the conditional expectation of a path functional $\phi\in L^1(\Theta;\mathbb{R})$ as
\begin{equation}\label{eq:mu_c_x}
\mu^u:=P\circ (X^u)^{-1},\hskip1ex \mu^{u,x}:= P^x\circ (X^u)^{-1},\hskip1ex E^{u,x}[\phi]=\int_\Theta \phi(\theta)\mu^{u,x}(d\theta).
\end{equation}
We shall refer to $\mu^{0,x}$ as the `reference measure', because all desired statistics shall defined with respect to $\mu^{0,x}$.

\subsection{Importance sampling and stochastic optimal control}
\label{ssec:stochastic_optimal_control_problem}

For a given path functional $W\in L^1(\Theta,\mu^{0,x};\mathbb{R})$, we wish to estimate the statistics of $W$ with respect to $\mu^{0,x}$. In the setting of importance sampling, an optimal importance sampling measure is any measure that yields a zero-variance estimator of the desired expected value. In our case, the optimal importance sampling measure solves the problem
\begin{equation}\label{eq:importance_sampling_problem}
 \text{min}\left\{ E^{u,x}\left[\left(\phi\frac{d\mu^0}{d\mu^u}-E^{u,x}\left[\phi\frac{d\mu^0}{d\mu^u}\right]\right)^2\right]\ \biggr\vert\ \left.\mu^u\right\vert_{\mathcal{F}_t} \ll \left.\mu^0\right\vert_{\mathcal{F}_t},\ \left.\mu^0\right\vert_{\mathcal{F}_t}\ll \left.\mu^u\right\vert_{\mathcal{F}_t}\hskip1ex\forall t>0\right\}.
\end{equation}
Thus, we seek to minimise the variance (with respect to $\mu^{u,x}$) of $\phi\tfrac{d\mu^0}{d\mu^u}$ over the set of measures $\mu^u$ that are mutually equivalent to $\mu^0$ when restricted to $\mathcal{F}_t$ for all $t>0$. Recall that if the control process $u'$ satisfies Kazamaki's or Novikov's condition, then the Radon-Nikodym derivative of $\mu^{u+u'}$ with respect to $\mu^u$ is a uniformly integrable exponential martingale with respect to $\mu^{u+u'}$, and the desired mutual equivalence holds, by the Cameron-Martin-Girsanov theorem on change of measure. The Radon-Nikodym derivative of $\mu^u$ with respect to $\mu^{u+u'}$ is 
\begin{equation}\label{eq:exponential_martingale}
 \left.\frac{d\mu^u}{d\mu^{u+u'}}\right\vert_{\mathcal{F}_t}=\mathcal{E}(M^{-u'})_t:=\exp\left[M^{-u'}_t-\frac{1}{2}\langle M^{-u'}\rangle_t\right]
\end{equation}
where $M^{u'}$ is a continuous local martingale with respect to $(\mathcal{F}_t)_{t\ge 0}$ and $\langle M^{u'}\rangle$ is the associated quadratic variation process, 
\begin{equation*}
M^{u'}_t:=\frac{1}{\sqrt{2\varepsilon}}\int_0^t u'_s dB_s,\hskip2ex
\langle M^{u'}\rangle_t:=\frac{1}{2\varepsilon}\int_0^t \vert u'_s\vert^2 ds.
\end{equation*}
By (\ref{eq:exponential_martingale}), we have the reweighting formula
\begin{equation}\label{eq:reweighting_formula}
 E^{u,x}\left[W\right]=E^{u+u',x}\left[W\mathcal{E}(M^{-u'})\right].
\end{equation}
A useful mnemonic device for applying the reweighting formula (\ref{eq:reweighting_formula}) is that the sum of the control superscripts on the left-hand side equals the sum of the control superscripts on right-hand side: $u=(u+u')+(-u')$. 

We now give some structure to the path functional of interest. Let $\Omega\subset\mathbb{R}^d$ be an open, bounded, and connected set, where every point on the boundary $\partial\Omega$ satisfies an exterior sphere condition. Define the first exit time of $X^u$ from $\Omega$ by
\begin{equation}\label{eq:first_exit_time}
 \tau(X^u):=\inf\left\{t>0\ \vert\ X^u_t\notin \Omega\right\}.
\end{equation}
Let $\kappa_r\in C(\Omega; [0,\infty))$ and $\kappa_t\in C(\partial\Omega;\mathbb{R})$. Define the path functional $W$ by
\begin{equation}\label{eq:work}
 W(X^u):=\int_0^{\tau(X^u)}\kappa_r(X^u_s)ds+\kappa_t\left(X^u_{\tau(X^u)}\right).
\end{equation}
We shall refer to $\kappa_r$ and $\kappa_t$ as the `running cost' and `terminal cost' respectively. Let $\sigma>0$ be fixed, and define the functions 
\begin{subequations}
 \begin{align}
   \psi(\sigma,x)&:=E^{0,x}\left[\exp(-\sigma W)\right]
   \label{eq:mgf}
   \\
   F(\sigma,x)&:=-\sigma^{-1}\log \psi(\sigma,x).
   \label{eq:cgf}
 \end{align}
\end{subequations}
Observe that (\ref{eq:cgf}) resembles Jarzynski's equality \cite{jarzynski_1997} in statistical physics, which relates the thermodynamic free energy difference between two equilibrium states of a statistical-mechanical system to an exponential average of nonequilibrium work. In the statistical physics context, $\sigma$ may be understood as an inverse temperature. Two important differences between the situation we consider and the situation considered by Jarzynski in \cite{jarzynski_1997} are that, in the former, the control force applied to the system is a function of the state and the control is applied up to a random stopping time, whereas in the latter, the control force is applied according to a preset schedule and over a deterministic time.

We wish to compute, for given $\sigma$ and $x$, the value of $F(\sigma,x)$. By (\ref{eq:mgf}), it suffices to compute $\psi(\sigma,x)$. Provided that $u$ satisfies Novikov's condition, we may combine (\ref{eq:mgf}) and (\ref{eq:reweighting_formula}) in order to obtain
\begin{equation}\label{eq:mgf_01}
 \psi(\sigma,x)=E^{u,x}\left[\exp(-\sigma W)\mathcal{E}(M^{-u})\right].
 \end{equation}
Thus, we may formulate the task of computing $\psi(\sigma,x)$ as an importance sampling problem of the form (\ref{eq:importance_sampling_problem}):
\begin{equation}\label{eq:importance_sampling_problem_parametrised_by_control}
 \text{min}\left\{E^{u,x}\left[\left(\exp(-\sigma W)\mathcal{E}(M^{-u})-\psi(\sigma,x)\right)^2\right]\ \biggr\vert\ \mathcal{E}(M^{-u})\text{ uniformly integrable}\right\}.
\end{equation}
In order to reformulate the importance sampling problem (\ref{eq:importance_sampling_problem_parametrised_by_control}) as a stochastic optimal control problem, we define
\begin{equation}\label{eq:ch3_nonequilibrium_estimator}
 K^{\sigma,u,\alpha}:=W+(2\sigma)^{-1}\langle M^u\rangle_\tau+\alpha\sigma^{-1}M^u_\tau,\quad\alpha\in\left\{0,1\right\}.
\end{equation}
By (\ref{eq:exponential_martingale}) and (\ref{eq:mgf_01}), it follows that $\psi(\sigma,x)=E^{u,x}\left[\exp\left(-\sigma K^{\sigma,u,1}\right)\right]$. For an arbitrary, adapted continuous local martingale $M$ and stopping time $T$, we shall write
\begin{equation*}
 M^T_t:=\begin{cases} 
        M_{T\wedge t} & \left\{0<T\right\} \\
        0 & \left\{0=T\right\}.
       \end{cases}
\end{equation*}
Since the diffusion matrix is a strictly positive multiple of the identity matrix, it holds that $\tau\in L^1(\mu^{u,x})$ (see \cite[Section 5.7, Lemma 7.4]{karatzas_shreve}). Thus, if $u\in L^\infty(\Omega)$, then it follows that $(M^u)^\tau$ is a $L^2(\mu^{u,x})$-bounded martingale, since 
\begin{equation*}
\lim_{t\to \infty} E^{u,x}[((M^u)^\tau_t)^2]=E^{u,x}[\langle (M^u)^\tau\rangle_\infty]\le \Vert u\Vert^2_{L^\infty} E^{u,x}[\tau].
\end{equation*}
Since $L^p$-boundedness for $p>1$ implies uniform integrability, we may apply the optimal stopping theorem to obtain $E^{u,x}[M^u_\tau]=0$ and $E^{u,x}[K^{\sigma,u,1}]=E^{u,x}[K^{\sigma,u,0}]$. By the strict concavity of the logarithm, Jensen's inequality and (\ref{eq:cgf}) yield
\begin{equation}\label{eq:optimisation_inequality}
 F(\sigma,x)\le E^{u,x}\left[ K^{\sigma,u,0}\right]=:\phi^{\sigma,x}(u)
\end{equation}
where inequality holds if and only if the random variable $K^{\sigma,u,1}$ is $\mu^{u,x}$-almost surely constant (see, e.g. \cite[Theorem 5]{mcshane1937}). Thus, we have shown that we can compute $F(\sigma,x)$ by solving the stochastic optimal control problem
\begin{equation}\label{eq:stochastic_optimal_control_problem}
  \min_{u\in\mathcal{U}} \ \phi^{\sigma,x}(u)\hskip2ex\text { subject to (2.1),}
\end{equation}
where $\mathcal{U}$ denotes the set of all previsible control processes. It was shown in \cite{hartmannschuette_efficient} that there exists a unique, Markovian optimal control $u^\sigma_{\text{opt}}$ given by
\begin{equation}\label{eq:optimal_control}
 u^{\sigma}_{\text{opt}}(x')=-2\varepsilon\sigma \nabla_x F(\sigma,x')\hskip1ex\forall x'\in\Omega.
\end{equation}
The link between the optimal control and importance sampling problems is evident from (\ref{eq:optimal_control}) since the optimal (zero-variance) importance sampling measure associated to $u^\sigma_{\text{opt}}$ is defined in terms of $F$.

\subsection{Connections to elliptic partial differential equations}
\label{ssec:connection_to_elliptic_pdes}

In this section, we recall some fundamental results concerning the functions $\psi(\sigma,\cdot)$ and $F(\sigma,\cdot)$. We shall consider second order partial differential operators of the form
\begin{equation}\label{eq:elliptic_partial_differential_operator}
 L:=\sum_{i,j=1}^d a_{ij} \frac{\partial}{\partial x_i}\frac{\partial}{\partial x_j}+\sum_{i=1}^d b_i\frac{\partial}{\partial x_i}+c
\end{equation}
where the functions $(a_{ij})_{i,j},(b_i)_i,c\in C(\Omega;\mathbb{R})$. Following the nomenclature from stochastic analysis (see \cite[Section 5.7A, Definition 7.1]{karatzas_shreve}), we shall say that $L$ is `uniformly elliptic' on $\Omega$ if there exists a constant $\lambda>0$ such that
\begin{equation}
 \lambda \vert \xi\vert^2\le \xi\cdot a\xi\hskip2ex\forall\xi\ne 0.
\end{equation}
Recall that the linear Dirichlet boundary value problem is given by
\begin{subequations}\label{eq:linear_elliptic_bvp}
\begin{align}
Lv(x)=f(x) \hskip1ex x&\in\Omega,
\label{eq:linear_elliptic_pde}
\\
v(y)=\varphi(y) \hskip1ex y&\in \partial\Omega,
\label{eq:linear_elliptic_bc}
\end{align}
\end{subequations}
for $f\in C(\Omega;\mathbb{R})$ and $\varphi\in C(\partial\Omega;\mathbb{R})$. Let $u$ in (\ref{eq:controlled_sde}) be an element of $C(\Omega;\mathbb{R}^d)$. Then the infinitesimal generator corresponding to the SDE (\ref{eq:controlled_sde}) given by
\begin{equation}\label{eq:infinitesimal_generator_controlled_sde}
 \mathcal{A}^u:=\varepsilon\Delta+\left(u-\nabla V\right)\cdot\nabla
\end{equation}
is a second order, uniformly elliptic partial differential operator. Define the linear Dirichlet boundary value problem
\begin{subequations}\label{eq:feynman_kac_bvp}
\begin{align}
\mathcal{A}^0 \psi(\sigma,x)-\sigma\kappa_r(x)\psi(\sigma,x)=0 \hskip1ex x&\in\Omega
\label{eq:feynman_kac_pde}
\\
\psi(\sigma,y)=\exp(-\sigma \kappa_t(y)) \hskip1ex y&\in \partial\Omega.
\label{eq:feynman_kac_bc}
\end{align}
\end{subequations}
We can rewrite (\ref{eq:feynman_kac_bvp}) as a problem of the form (\ref{eq:linear_elliptic_bvp}), with $L:=\mathcal{A}^0-\sigma\kappa_r$ and $f=0$ on $\Omega$, and with $\varphi=\exp(-\sigma\kappa_t)$ on $\partial\Omega$. The following classical result establishes (\ref{eq:mgf}) as the `Feynman-Kac formula' for the solution to (\ref{eq:feynman_kac_bvp}):
\begin{myprop}\label{prop:feynman_kac_representation}
If
\begin{enumerate}
 \item[(i)] $\Omega$ is an open, bounded domain,
 \item[(ii)] the first-order coefficient $-\nabla V$ satisfies the growth conditions (\ref{eq:growth_conditions}), and
\item[(iii)] $\kappa_r\in C(\overline{\Omega};[0,\infty))$ and $\kappa_t\in C(\partial\Omega;\mathbb{R})$,
\end{enumerate}
then there exists a unique solution $\psi(\sigma,\cdot)\in C^2(\Omega;\mathbb{R})\cap C(\overline{\Omega};\mathbb{R})$ that solves (\ref{eq:feynman_kac_bvp}), and $\psi(\sigma,\cdot)$ is of the form (\ref{eq:mgf}).
\end{myprop}
\begin{proof}
  See \cite[Section 5.7, Proposition 7.2]{karatzas_shreve} and \cite[Section 5.7, Lemma 7.4]{karatzas_shreve}.
\end{proof}

Proposition \ref{prop:feynman_kac_representation} is important in establishing a link between the function $\psi(\sigma,\cdot)$ defined in (\ref{eq:mgf}) and the linear Dirichlet boundary value problem (\ref{eq:feynman_kac_bvp}). In \S\ref{sec:variational_analysis}, we shall require more regularity on the solution $\psi(\sigma,\cdot)$ in order to show $L^2$-boundedness of exponential martingales. Since we will also work with continuous functions that are bounded on the compact set $\overline{\Omega}$, we also wish to know when the optimal control $u^\sigma_{\text{opt}}$ defined in (\ref{eq:optimal_control}) is bounded on $\overline{\Omega}$. Recall that a bounded domain $\Omega\subset\mathbb{R}^d$ is of class $C^{k,\alpha}$ for $k\in\mathbb{N}_0$ and $0\le \alpha\le 1$ if $\partial\Omega$ is the graph of a $C^{k,\alpha}$ function of $d-1$ coordinates \cite[Section 6.2]{gilbarg_trudinger}. 
\begin{mythm}\label{thm:feynman_kac_representation_uniformly_elliptic_case_uniformly_holder}
Let $0\le \alpha\le 1$ and $L$ be uniformly elliptic. If 
\begin{enumerate}
 \item[(i)] $\Omega$ is an open, bounded, $C^{2,\alpha}$ domain,
 \item[(ii)] the coefficients of $L$ and $f$ in (\ref{eq:linear_elliptic_pde}) belong to $C^{\alpha}(\overline{\Omega})$, 
 \item[(iii)]the zeroth-order coefficient of $L$ is nonpositive, i.e. $c\le 0$ on $\Omega$, and
 \item[(iv)] $\kappa_t \in C^{2,\alpha}(\overline{\Omega})$,
\end{enumerate}
then the Dirichlet problem (\ref{eq:linear_elliptic_bvp}) has a unique solution $v\in C^{2,\alpha}(\overline{\Omega})$.
\end{mythm}
\begin{proof}
 See \cite[Section 6.3, Theorem 6.14]{gilbarg_trudinger}.
\end{proof}

By the definitions (\ref{eq:mgf}) and (\ref{eq:cgf}), it holds that $\psi(\sigma,\cdot)=\exp(-\sigma F(\sigma,\cdot))$. By the chain rule from ordinary calculus, we can formally transform the linear Dirichlet problem (\ref{eq:feynman_kac_bvp}) into the nonlinear Dirichlet problem 
\begin{subequations}\label{eq:hamilton_jacobi_bellman_bvp}
\begin{align}
\mathcal{A}^0 F(\sigma,x)-\varepsilon\sigma \vert \nabla_x F(\sigma,x)\vert^2+\kappa_r(x)=0 \hskip1ex x&\in\Omega
\label{eq:hamilton_jacobi_bellman_pde}
\\
F(\sigma,y)=\kappa_t(y) \hskip1ex y&\in \partial\Omega,
\label{eq:hamilton_jacob_bellman_bc}
\end{align}
\end{subequations}
where (\ref{eq:hamilton_jacobi_bellman_pde}) is the Hamilton-Jacobi-Bellman (HJB) equation of the stochastic optimal control problem (\ref{eq:stochastic_optimal_control_problem}). If there exists a unique solution $F(\sigma,\cdot)$ of (\ref{eq:hamilton_jacobi_bellman_bvp}), then it is the value function of the stochastic optimal control problem. In general, there are no classical solutions to (\ref{eq:hamilton_jacobi_bellman_bvp}), and one must search for a viscosity solution. Although viscosity solutions are interesting objects in their own right, our present interest pertains solely to classical solutions. For completeness, we shall consider sufficient conditions for which there exists a classical solution to the nonlinear Dirichlet problem (\ref{eq:hamilton_jacobi_bellman_bvp}).

Let $\Gamma:=\Omega\times\mathbb{R}\times\mathbb{R}^d\times\mathbb{R}^{d\times d}$, $\gamma=(x,z,p,r)\in \Gamma$ be arbitrary, and $\mathcal{H}:\Gamma\to \mathbb{R}$ be defined by
 \begin{equation}\label{eq:hamilton_jacobi_map}
  \mathcal{H}(x,z,p,r):=\varepsilon\left(\sum_i r_{ii}\right)-\nabla V(x)\cdot p-\varepsilon\sigma\vert p\vert^2+\kappa_r(x).
 \end{equation}
From (\ref{eq:hamilton_jacobi_map}), it follows that $\mathcal{H}$ is constant with respect to $z$, concave with respect to $p$, and linear with respect to $r$. Note that we can rewrite (\ref{eq:hamilton_jacobi_bellman_pde}) as
\begin{equation*}
 \mathcal{H}(x,F(\sigma,x),\nabla_x F(\sigma,x),\nabla^2_x F(\sigma,x))=0\quad x\in \Omega.
\end{equation*}
We have
\begin{mythm}\label{thm:classical_solution_of_hamilton_jacobi_bellman_bvp}
 Let $\mathcal{H}$ be defined as in (\ref{eq:hamilton_jacobi_map}).
 \begin{enumerate}
  \item[(i)] If $\Omega$ is a bounded domain satisfying an exterior sphere condition at each boundary point, and
  \item[(ii)] $\nabla V\in C^2(\Omega;\mathbb{R}^d)$ and $\kappa_r\in C^2(\Omega;\mathbb{R})$,
 \end{enumerate}
then the classical Dirichlet problem (\ref{eq:hamilton_jacobi_bellman_bvp}) is uniquely solvable in $C^2(\Omega;\mathbb{R})\cap C(\overline{\Omega};\mathbb{R})$ for any terminal cost  $\kappa_t\in C(\partial\Omega;\mathbb{R})$.
\end{mythm}
\begin{proof}
 The proof follows from an application of \cite[Section 17.5, Theorem 17.17]{gilbarg_trudinger}. Therefore, we must verify that the assumptions of that result hold for the operator $\mathcal{H}$ defined in (\ref{eq:hamilton_jacobi_map}). The assumptions of the present theorem guarantee that $\mathcal{H}\in C^2(\Gamma)$, and that $\mathcal{H}$ is constant with respect to $z$, concave with respect to $p$, and linear with respect to $r$. Let $\mathcal{H}_{ij}(x,z,p,r):=\frac{\partial\mathcal{H}}{\partial r_{ij}}(x,z,p,r)$. It remains to verify that the structure conditions (see \cite[Section 17.5, Eq. (17.53)]{gilbarg_trudinger}) 
 \begin{subequations}\label{eq:structural_conditions}
\begin{align}
   0<\lambda \vert \xi\vert^2\le \sum_{i,j}\mathcal{H}_{ij}\xi_i\xi_j\le \Lambda \vert \xi\vert^2
   \label{eq:structural_conditions_01}
   \\
   \vert \mathcal{H}_p\vert,\vert\mathcal{H}_z\vert,\vert\mathcal{H}_{rx}\vert,\vert\mathcal{H}_{px}\vert,\vert\mathcal{H}_{zx}\vert\le \mu\lambda,
   \label{eq:structural_conditions_02}
   \\
   \vert\mathcal{H}_x\vert,\vert\mathcal{H}_{xx}\vert\le \mu\lambda(1+\vert p\vert+\vert r\vert),
   \label{eq:structural_conditions_03}
  \end{align} 
\end{subequations}
 are satisfied for nonzero $\xi\in\mathbb{R}^d$, arbitrary $(x,z,p,r)\in \Gamma$, a function $\lambda$ that is nonincreasing in $\vert z\vert$, and functions $\Lambda$ and $\mu$ that are nondecreasing in $\vert z\vert$. By (\ref{eq:hamilton_jacobi_map}), condition (\ref{eq:structural_conditions_01}) holds with $\lambda=\Lambda=\varepsilon$ since $\mathcal{H}_{ij}=\delta_{ij}\varepsilon$. Furthermore, $\mathcal{H}_z$ and $\mathcal{H}_{rx}$ vanish everywhere, and 
\begin{subequations}\label{eq:checking_structural_conditions_part_01}
 \begin{align}
  \frac{\partial\mathcal{H}}{\partial p_i}(x,z,p,r)&=-\frac{\partial V}{\partial x_i}(x)-2\varepsilon\sigma p_i,
  \\
  \frac{\partial^2 \mathcal{H}}{\partial p_i\partial x_j}(x,z,p,r)&=-\frac{\partial^2 V}{\partial x_i x_j}(x),
  \\
    \frac{\partial\mathcal{H}}{\partial x_i}(x,z,p,r)&=\frac{\partial }{\partial x_i}\left(\kappa_r(x)-\nabla V(x)\cdot p\right),
  \\
  \frac{\partial^2 \mathcal{H}}{\partial x_i\partial x_j}(x,z,p,r)&=\frac{\partial^2 }{\partial x_i x_j}\left(\kappa_r(x)-\nabla V(x)\cdot p\right).
 \end{align}
\end{subequations}
Since neither of the right-hand sides of (\ref{eq:checking_structural_conditions_part_01}) depends on $\vert z\vert$, we may define $\mu:\Gamma\to\mathbb{R}$ to be any function that satisfies (\ref{eq:structural_conditions_02})--(\ref{eq:structural_conditions_03}).
\end{proof}

Note that we may generalise Theorem \ref{thm:classical_solution_of_hamilton_jacobi_bellman_bvp} by weakening the requirement (ii), i.e. by only requiring that $\nabla V$ and $\kappa_r$ be twice weakly differentiable functions of the spatial variable $x$ and interpreting the structure conditions (\ref{eq:structural_conditions}) in the weak sense \cite[Chapter 17, Problem 17.4]{gilbarg_trudinger}.

Theorem \ref{thm:feynman_kac_representation_uniformly_elliptic_case_uniformly_holder} may be generalised to larger classes of domains that need not be bounded (see \cite[Section 6.3, Theorem 6.13]{gilbarg_trudinger} and the notes at the end of \cite[Chapter 6]{gilbarg_trudinger}). Note that even if the zeroth-order coefficient of the operator $L$ fails to satisfy the nonpositivity condition, solutions of (\ref{eq:linear_elliptic_bvp}) still belong to $C^{k+2,\alpha}(\overline{\Omega})$ for $k\in\mathbb{N}_0$, provided that the domain is $C^{k+2,\alpha}$, the boundary function belongs to $C^{k+2,\alpha}(\overline{\Omega})$, and the coefficients of the uniformly elliptic operator $L$ belong to $C^{k,\alpha}(\overline{\Omega})$. However, uniqueness of solutions no longer holds \cite[Section 6.4, Theorem 6.19]{gilbarg_trudinger}. 

\section{Variational analysis of the control functional}
\label{sec:variational_analysis}

In this section, we derive expressions for the first variation and the second variation of the control functional with respect to admissible perturbations of the feedback control function. In Lemma \ref{lem:finiteness_of_mgf_of_tau}, we adapt an argument due to Gilbarg and Trudinger, and use Fredholm theory in order to assert that the moment generating function of the first exit time defined in (\ref{eq:first_exit_time}) is finite over any open, bounded interval containing the origin. We then use Lemma \ref{lem:finiteness_of_mgf_of_tau} to show that Kazamaki's sufficient condition (\ref{eq:suffcon_Lq_boundedness_of_stopped_exp_MG}) for $L^q$-boundedness of exponential martingales holds for the case that $q>1$. Together with the Cauchy-Schwarz inequality and Lebesgue's dominated convergence theorem, $L^q$-boundedness yields the expressions for the first and second variations. The main result of this section, Theorem \ref{thm:uniform_lower_bound_for_W_sufficient_cond_for_convexity_undiscretised_functional}, shows the strict convexity of the control functional by proving that the second variation is a positive definite functional on the Banach space $C^{1,\alpha}(\overline{\Omega};\mathbb{R}^d)$ whenever the work path functional $W$ defined in (\ref{eq:work}) is bounded from below. Recall that $\sigma>0$ is a fixed parameter. All the results in this section, unless otherwise stated, shall be based upon the following
\begin{myasmp}\label{asmp:regularity_asmp_dV_domain}
Let $0\le \alpha\le 1$. It holds that
\begin{enumerate}
 \item[(i)] the domain $\Omega$ is bounded and of class $C^{2,\alpha}$, 
 \item[(ii)] the gradient of the potential $\nabla V\in C^{\alpha}(\overline{\Omega};\mathbb{R}^d)\cap W^2(\Omega;\mathbb{R}^d)$, 
 \item[(iii)] the running cost $\kappa_r\in C^{\alpha}(\overline{\Omega};[0,\infty))\cap W^2(\Omega;\mathbb{R})$, and 
 \item[(iv)] the terminal cost $\kappa_t\in C^{2,\alpha}(\overline{\Omega};\mathbb{R})$.
\end{enumerate}
\end{myasmp}

By Theorem \ref{thm:feynman_kac_representation_uniformly_elliptic_case_uniformly_holder}, there exists a unique solution to (\ref{eq:linear_elliptic_bvp}) in $C^{2,\alpha}(\overline{\Omega};\mathbb{R})$ with $f=0$ on $\Omega$ and $\varphi=\exp(-\sigma\kappa_t)$ on $\partial\Omega$. Proposition \ref{prop:feynman_kac_representation} guarantees that the unique solution is described by the Feynman-Kac formula (\ref{eq:mgf}). It follows from (\ref{eq:mgf}) and Jensen's inequality that $\exp(E^{0,x}[-\sigma W])\le \psi(\sigma,x)$; since finiteness of $\psi(\sigma,x)$ on $\overline{\Omega}$ implies finiteness of $E^{0,x}[-\sigma W]$, it follows that $\psi(\sigma,x)$ is bounded away from zero on $\overline{\Omega}$. On the other hand, the maximum principle implies that $\psi(\sigma,x)$ is bounded from above. Thus, given the relation (\ref{eq:cgf}), it follows that $F(\sigma,\cdot)\in C^{2,\alpha}(\overline{\Omega};\mathbb{R})$. By Theorem \ref{thm:classical_solution_of_hamilton_jacobi_bellman_bvp}, $F(\sigma,\cdot)$ is a classical solution of the HJB problem (\ref{eq:hamilton_jacobi_bellman_bvp}), and the optimal control belongs to $C^{1,\alpha}(\overline{\Omega};\mathbb{R}^d)$.

\subsection{$L^q$-bounded exponential martingales}
\label{ssec:Lq_bounded_exponential_martingales}
The first result below is of crucial importance:
\begin{mylem}\label{lem:finiteness_of_mgf_of_tau}
 Suppose that Assumption \ref{asmp:regularity_asmp_dV_domain} holds. If $u\in C^\alpha(\overline{\Omega})$, then for all $r>0$, there exists $\sigma_p>r$ such that $E^{u,x}\left[\exp(\sigma_p\tau)\right]$ is finite for all $x\in\overline{\Omega}$.
\end{mylem}\begin{proof}
 We prove the statement using Fredholm theory, following the argument outlined at the end of \cite[Section 6.3]{gilbarg_trudinger}. Observe that, for an arbitrary $\sigma_n<0$, $E^{u,x}[\exp(\sigma_n\tau)]$ uniquely solves the Feynman-Kac boundary value problem (\ref{eq:feynman_kac_bvp}) with the choice of $-\sigma=\sigma_n<0$, $\kappa_r=1$ and $\kappa_t=0$. Since $\kappa_r$ and $\kappa_t$ are constant, they may be extended to all of $\overline{\Omega}$. Therefore, we may conclude from Theorem \ref{thm:feynman_kac_representation_uniformly_elliptic_case_uniformly_holder} that $E^{u,x}[\exp(\sigma_n\tau)]$ belongs to $C^{2,\alpha}(\overline{\Omega};\mathbb{R})$. We rewrite this Feynman-Kac boundary value problem as an elliptic boundary value problem of the form (\ref{eq:linear_elliptic_bvp}):
 \begin{subequations}\label{eq:elliptic_partial_differential_operator_for_stopping_time}
 \begin{align}
  L^{u,\sigma_n}v(x)&:=\mathcal{A}^u v(x)+\sigma_n v(x)=0,\hskip1ex x\in \Omega,
  \label{eq:elliptic_partial_differential_operator_for_stopping_time_pde}
  \\
  v(y)&=1,\hskip1ex y\in\partial\Omega.
  \label{eq:elliptic_partial_differential_operator_for_stopping_time_bc}
  \end{align}
 \end{subequations}
Since all the coefficients of $L^{u,\sigma_n}$ belong to $C^{\alpha}(\overline{\Omega})$, the operator $L$ defines a closed map from coefficients that belong to $C^\alpha(\overline{\Omega})$ to solutions that belong to $C^{2,\alpha}(\overline{\Omega})$. By \cite[Theorem 5.5]{gilbarg_trudinger}, it follows that there is a countable, discrete set $\Sigma\subset\mathbb{R}$, such that if $\sigma\in \Sigma$, then the homogeneous problem 
 \begin{subequations}\label{eq:homogeneous_problem_01}
 \begin{align}
  L^{u,\sigma_n-\sigma} v(x)&=L^{u,\sigma_n}v(x)-\sigma v(x)=0, \hskip1ex x\in\Omega,
  \label{eq:homogeneous_problem_01_pde}
  \\
  v(y)&=0,\hskip1ex y\in\partial\Omega
  \label{eq:homogeneous_problem_01_bc}
 \end{align}
 \end{subequations}
has nontrivial solutions. Therefore, for $\sigma\notin\Sigma$, the homogeneous problem (\ref{eq:homogeneous_problem_01}) has only the trivial solution $v\equiv 0$. By the Fredholm alternative (see \cite[Section 6.3, Theorem 6.15]{gilbarg_trudinger}), it follows that the corresponding inhomogeneous problem obtained by replacing (\ref{eq:homogeneous_problem_01_bc}) with (\ref{eq:elliptic_partial_differential_operator_for_stopping_time_bc}) has a unique solution in $C^{2,\alpha}(\overline{\Omega})$. By the extreme value theorem and compactness of $\overline{\Omega}$, the unique solution is finite on $\overline{\Omega}$. By It\^{o}'s formula, the unique solution admits the representation $v(x)=E^{u,x}\left[\exp((\sigma_n-\sigma)\tau)\right]$. Now let $r>0$ be arbitrary. Since $\Sigma$ is discrete, there exists a $\sigma\in(-\infty,0)\cap \Sigma^\complement$ such that $\sigma_p:=\sigma_n-\sigma>r$, and the conclusion follows.
\end{proof}
\begin{mycor}\label{cor:finiteness_of_mgf_of_tau_over_interval}
  For all $x\in\overline{\Omega}$, $\sigma_n<0$ and $r>0$, there exists a closed interval $[\sigma_n,\sigma_p]$ with $\sigma_p>r$ such that the moment generating function defined by $E^{u,x}\left[\exp(\sigma\tau)\right]$ is finite for all $\sigma\in [\sigma_n,\sigma_p]$.
\end{mycor}
\begin{proof}
Let $x\in\overline{\Omega}$ be arbitrary. For any $\sigma_n<0$, $E^{u,x}[\exp(\sigma_n\tau)]\le 1$ by monotonicity of the integral. By Lemma \ref{lem:finiteness_of_mgf_of_tau}, there exists a $\sigma_p>r$ such that $E^{u,x}[\exp(\sigma_p\tau)]$ is finite. For arbitrary $\sigma_n<\sigma<\sigma_p$, there exists a $0<\theta<1$ such that $\sigma=\theta\sigma_n+(1-\theta)\sigma_p$. By the strict convexity of the exponential, we have
 \begin{equation*}
  \exp(\sigma\tau)<\theta \exp(\sigma_n\tau)+(1-\theta)\exp(\sigma_p\tau).
 \end{equation*}
Taking expectations with respect to $\mu^{u,x}$ yields the desired conclusion.
\end{proof}

\begin{mythm}\label{thm:suffcon_Lq_boundedness_of_stopped_exp_MG}
 Let $\tau$ be an arbitrary stopping time, and let $\sigma>0$ be such that $E\left[\exp(\sigma\tau)\right]$ is finite. If for some $1<p<\infty$, it holds that $\varphi\in L^\infty(\Omega)$ satisfies
 \begin{equation}\label{eq:suffcon_Lq_boundedness_of_stopped_exp_MG}
  \Vert \varphi\Vert^2_{L^\infty(\Omega)}\le 2\sigma\frac{(\sqrt{p}-1)^2}{p},
 \end{equation}
then the stopped exponential martingale $\mathcal{E}((M^\varphi)^\tau)$ is an $L^q$-bounded martingale for $q^{-1}=1-p^{-1}$ and $E[\mathcal{E}(M^\varphi)^q_\tau]$ is finite.
\end{mythm}\begin{proof}
Let $p$ and $q$ be a pair of H\"{o}lder conjugate exponents. Recall \cite[Section 1.2, Theorem 1.5]{kazamaki}: for a given continuous local martingale $M$, if
\begin{equation}\label{eq:kazamakis_condition_for_Lq_boundedness_of_exp_MG}
 \sup\left\{E\left[\exp\left(\frac{1}{2}\frac{\sqrt{p}}{\sqrt{p}-1}M_T\right)\right]\ \biggr\vert\ T\text{ a bounded stopping time }\right\}<\infty,
\end{equation}
then the exponential martingale $\mathcal{E}(M)$ is $L^q$-bounded. Let $T$ be an arbitrary bounded stopping time. Since it follows from (\ref{eq:exponential_martingale}) that
\begin{equation*}
 \exp\left(\frac{1}{2} M_T\right)=\left(\mathcal{E}(M)_T\right)^{1/2}\exp\left(\frac{1}{4}\langle M\rangle_T\right),
\end{equation*}
it holds by the Cauchy-Schwarz inequality that
\begin{equation}\label{eq:cauchy_schwarz_inequality_application_to_exp_of_MG}
 E\left[\exp\left(\frac{1}{2}M_T\right)\right]\le E\left[\exp\left(\frac{1}{2}\langle M\rangle_T\right)\right]^{1/2}.
\end{equation}
Combining the inequality (\ref{eq:cauchy_schwarz_inequality_application_to_exp_of_MG}) with (\ref{eq:kazamakis_condition_for_Lq_boundedness_of_exp_MG}), we therefore obtain that 
\begin{equation}\label{eq:kazamakis_condition_for_Lq_boundedness_of_exp_MG_in_terms_of_QV}
 \sup\left\{E\left[\exp\left(\frac{p}{2(\sqrt{p}-1)^2}\langle M\rangle_T\right)\right]\ \biggr\vert\ T\text{ a bounded stopping time }\right\}<\infty
\end{equation}
is a sufficient condition for the $L^q$-boundedness of $\mathcal{E}(M)$. By the assumption that $\varphi\in L^\infty(\Omega)$, it holds that
\begin{equation*}
\langle (M^\varphi)^\tau\rangle_T\le \Vert \varphi\Vert^2_{L^\infty(\Omega)}\left(\tau\wedge T\right).
\end{equation*}
Since (\ref{eq:suffcon_Lq_boundedness_of_stopped_exp_MG}) is equivalent to
\begin{equation*}
 \frac{p}{2(\sqrt{p}-1)^2}\Vert \varphi\Vert^2_{L^\infty(\Omega)}\le\sigma,
\end{equation*}
and since $E[\exp(\sigma\tau)]$ is finite, it follows that (\ref{eq:kazamakis_condition_for_Lq_boundedness_of_exp_MG_in_terms_of_QV}) is satisfied, by the monotonicity of the exponential and of the integral. Therefore $\mathcal{E}(M^\varphi)^\tau$ is $L^q$-bounded. By the definition of $L^q$-boundedness, it follows that $\mathcal{E}(M^\varphi)_\tau\in L^q(\mu^{u,x})$, as desired.
\end{proof}

The next result combines Lemma \ref{lem:finiteness_of_mgf_of_tau} and Theorem \ref{thm:suffcon_Lq_boundedness_of_stopped_exp_MG} to show that, for any $\varphi\in L^\infty(\Omega;\mathbb{R}^d)$, it holds that $\mathcal{E}((M^\varphi)^\tau)$ is $L^q(\mu^{u,x})$-bounded for any $q>1$.
\begin{mycor}\label{cor:bounded_perturbations_yield_Lq_bounded_exp_MGs}
 Let $\tau$ be defined as in (\ref{eq:first_exit_time}) and suppose that Assumption \ref{asmp:regularity_asmp_dV_domain} holds. If $u\in C^\alpha(\overline{\Omega})$ and $\varphi\in L^\infty(\Omega)$, then 
 \begin{enumerate}
  \item[(i)] $\mathcal{E}(M^{\varphi})^\tau$ is $L^q(\mu^{u,x})$-bounded for all $q>1$, and
  \item[(ii)] for any $1\le r<\infty$, $E^{u,x}[\langle M^\varphi\rangle_\tau^r]$ is finite, and for any $0<s<\infty$, $E^{u,x}[\sup_{0\le t\le \tau}\vert M^\varphi_t\vert^s]$ is finite.
 \end{enumerate}
\end{mycor}
\begin{proof}
 We first prove statement (i). Let $p$ be the H\"{o}lder conjugate of $q$, and define 
 \begin{equation*}
  r_p:=\frac{p}{2(\sqrt{p}-1)^2}\Vert \varphi\Vert_{L^\infty(\Omega)}^2.
 \end{equation*}
 Since $u\in C^\alpha(\overline{\Omega})$, there exists a $\sigma_p>r_p$ so that $E^{u,x}\left[\exp(\sigma_p\tau)\right]$ is finite for all $x\in\overline{\Omega}$, by Lemma \ref{lem:finiteness_of_mgf_of_tau}. Hence, $\varphi$ satisfies (\ref{eq:suffcon_Lq_boundedness_of_stopped_exp_MG}), and the conclusion follows by Theorem \ref{thm:suffcon_Lq_boundedness_of_stopped_exp_MG}.
 
 To prove statement (ii), observe that (\ref{eq:kazamakis_condition_for_Lq_boundedness_of_exp_MG_in_terms_of_QV}) implies that $E[\langle M^{\varphi}\rangle^r_\tau]$ is finite for all $r\in\mathbb{N}$, and hence for all $1\le r<\infty$, by Jensen's inequality. The second assertion of statement (ii) follows from this observation and the Burkholder-Davis-Gundy inequalities \cite[Chapter IV, \S 4, Theorem 4.1]{revuz_yor}. 
\end{proof}

The next result shows that Assumption \ref{asmp:regularity_asmp_dV_domain} suffices to guarantee that, for bounded controls $u$, the quantity $E^{u,x}\left[\exp(-\sigma W)\right]$ is a finite for all $x\in\overline{\Omega}$.
\begin{mylem}\label{lem:even_moments_of_W_wrt_mu_u_x_are_finite}
 Let $W$ be defined as in (\ref{eq:work}) and suppose that Assumption \ref{asmp:regularity_asmp_dV_domain} holds. If $u\in L^\infty(\Omega)$, then $E^{u,x}\left[\exp(-\sigma W)\right]$ is finite, and $W\in L^{n}(\mu^{u,x})$ for all $n\in\mathbb{N}$.
\end{mylem}\begin{proof}
Since the constant $0$ belongs to $C^\alpha(\overline{\Omega})$, and since we assumed that $u\in L^\infty(\Omega)$, we may apply Corollary \ref{cor:bounded_perturbations_yield_Lq_bounded_exp_MGs} to assert that $\mathcal{E}(M^u)^\tau$ is $L^2(\mu^{0,x})$-bounded. By the reweighting formula (\ref{eq:reweighting_formula}) and the Cauchy-Schwarz inequality, 
\begin{equation*}
 E^{u,x}\left[\exp(-\sigma W)\right]\le\left(E^{0,x}\left[\exp(-2\sigma W)\right] E^{0,x}\left[\mathcal{E}(M^u)_\tau^2\right]\right)^{1/2},
\end{equation*}
where Theorem \ref{thm:feynman_kac_representation_uniformly_elliptic_case_uniformly_holder} implies that the first term on the right-hand side is finite and $L^2(\mu^{0,x})$-boundedness of $\mathcal{E}(M^u)\tau$ implies that the second term on the right-hand side is finite. This proves the first conclusion. The second conclusion follows from the series expansion of the exponential.
\end{proof}

The next result of this section will be useful later and follows from statement (ii) of Corollary \ref{cor:bounded_perturbations_yield_Lq_bounded_exp_MGs} and Lemma \ref{lem:even_moments_of_W_wrt_mu_u_x_are_finite}.
\begin{mycor}\label{cor:nonequilibrium_estimator_moments_of_all_orders}
 Suppose that Assumption \ref{asmp:regularity_asmp_dV_domain} holds. If $\sigma>0$ and $u\in C^\alpha(\overline{\Omega})$, then $K^{\sigma,u,0}, K^{\sigma,u,1}\in L^n(\mu^{u,x})$ for all $n\in\mathbb{N}$.
\end{mycor}

\subsection{Convergence lemma}
\label{ssec:convergence_lemma}

By the definition (\ref{eq:exponential_martingale}) of the exponential martingale, it holds that for $0\ne \delta\in\mathbb{R}$, a Borel-measurable function $\varphi:\mathbb{R}^d\to\mathbb{R}^d$, and a stopping time $\tau$,
\begin{equation*}
  \mathcal{E}(M^{\delta\varphi})_\tau=\exp\left[M^{\delta\varphi}_\tau- \frac{1}{2}\langle M^{\delta\varphi}\rangle_\tau\right]=\exp\left[\delta\left(M^{\varphi}_\tau-\frac{\delta}{2}\langle M^{\varphi}\rangle_\tau\right)\right].
 \end{equation*}
The series expansion of the exponential thus yields
 \begin{equation}\label{eq:D_minus_1}
  \mathcal{E}(M^{\delta\varphi})_\tau=1+\sum^\infty_{m=1}\frac{\delta^m\left(M^\varphi_\tau-\delta 2^{-1}\langle M^\varphi\rangle_\tau\right)^m}{m!}
 \end{equation}
 and hence 
\begin{equation}\label{eq:D_minus_1_over_delta_minus_M}
 \frac{\mathcal{E}(M^{\delta\varphi})_\tau-1}{\delta}-M^{\varphi}_\tau=-\frac{\delta}{2}\langle M^\varphi\rangle_\tau+\sum^\infty_{m=2}\frac{\delta^{m-1}\left(M^\varphi_\tau-\delta 2^{-1}\langle M^\varphi\rangle_\tau\right)^m}{m!}.
\end{equation}
We have the following convergence result:
\begin{mylem}\label{lem:dominated_convergence_theorem_application}
 Suppose that Assumption \ref{asmp:regularity_asmp_dV_domain} holds, that $u\in C^\alpha(\overline{\Omega})$, and $\varphi\in L^\infty(\Omega)$. Then for any nonzero sequence $(\delta_n)_n$ satisfying $\delta_n\to 0$ as $n\to\infty$, it holds that $\mathcal{E}(M^{\delta_n\varphi})_\tau\in L^2(\mu^{u,x})$ for all $n\in\mathbb{N}$, and
 \begin{align}\label{eq:convergence_of_expMG_to_1}
  \lim_{n\to\infty}E^{u,x}\left[\left(\mathcal{E}(M^{\delta_n\varphi})_\tau-1\right)^2\right]&=0,\\
  \label{eq:convergence_of_expMG_minus_1_over_delta}
  \lim_{n\to\infty}E^{u,x}\left[\left(\frac{\mathcal{E}(M^{\delta_n\varphi})_\tau-1}{\delta_n}-M^{\varphi}_\tau\right)^2\right]&=0.
 \end{align}
\end{mylem}\begin{proof}
Without loss of generality, we may assume that the sequence $(\vert\delta_n\vert)_n$ is bounded by $0<r<1$, so that
\begin{equation}\label{eq:bound_on_sequence_of_perturbations}
 \Vert \delta_n\varphi\Vert_{L^\infty(\Omega)}\le r\Vert \varphi\Vert_{L^\infty(\Omega)}<\Vert \varphi\Vert_{L^\infty(\Omega)}.
\end{equation}
Therefore, for all $n\in\mathbb{N}$, $\delta_n\varphi\in L^\infty(\Omega)$, and by Corollary \ref{cor:bounded_perturbations_yield_Lq_bounded_exp_MGs}, it follows that $\mathcal{E}(M^{\delta_n\varphi})_\tau\in L^2(\mu^{u,x})$ for all $n\in\mathbb{N}$. By H\"{o}lder's inequality, we have (see the proof of \cite[Section 1.2, Theorem 1.5]{kazamaki})
\begin{equation*}
 E^{u,x}\left[\mathcal{E}(M^{\delta_n\varphi})^2_\tau\right]\le E^{u,x}\left[\exp\left(\frac{\sqrt{2}}{2(\sqrt{2}-1)}M^{\delta_n\varphi}_\tau\right)\right]^{2/(\sqrt{2}+1)}.
\end{equation*}
Applying (\ref{eq:cauchy_schwarz_inequality_application_to_exp_of_MG}) and (\ref{eq:bound_on_sequence_of_perturbations}) to the inequality above yields
\begin{align}
 E^{u,x}\left[\mathcal{E}(M^{\delta_n\varphi})^2_\tau\right]&\le E^{u,x}\left[\exp\left(\frac{1}{(\sqrt{2}-1)^2}\Vert \delta_n\varphi\Vert^2_{L^\infty(\Omega)}\tau\right)\right]^{1/(\sqrt{2}+1)}
 \nonumber\\
 &< E^{u,x}\left[\exp\left(\frac{1}{(\sqrt{2}-1)^2}\Vert \varphi\Vert^2_{L^\infty(\Omega)}\tau\right)\right]^{1/(\sqrt{2}+1)}.
\label{eq:L2_bounded_expMGs}
\end{align}
By Corollary \ref{cor:finiteness_of_mgf_of_tau_over_interval}, it follows that the collection of random variables $(\mathcal{E}(M^{\delta_n\varphi})_\tau)_n$ is bounded in $L^2(\mu^{u,x})$. We can therefore apply Lebesgue's dominated convergence theorem and the identity (\ref{eq:D_minus_1}) in order to obtain
\begin{equation*}
 \lim_{n\to\infty}E^{u,x}\left[\left(\mathcal{E}(M^{\delta_n\varphi})_\tau-1\right)^2\right]=E^{u,x}\left[ \lim_{n\to\infty}\left(\mathcal{E}(M^{\delta_n\varphi})_\tau-1\right)^2\right]=0,
\end{equation*}
which proves (\ref{eq:convergence_of_expMG_to_1}). To prove (\ref{eq:convergence_of_expMG_minus_1_over_delta}), we first observe that by (\ref{eq:D_minus_1}), the random series
\begin{equation}\label{eq:D_minus_1_squared}
  R(\delta_n):=\left(\mathcal{E}(M^{\delta_n\varphi})_\tau-1\right)^2-\delta_n^2\left(M^{\varphi}_\tau\right)^2\in L^1(\mu^{u,x})
 \end{equation}
consists of terms that are cubic or higher in $\delta_n$, so that $\lim_{n\to\infty} \delta_n^{-2}E^{u,x}[R(\delta_n)]=0$. Then there exists some constant $C(r)>0$ so that
\begin{align}
 E^{u,x}\left[ \left(\frac{\mathcal{E}(M^{\delta_n\varphi})_\tau-1}{\delta_n}\right)^2\right]&=E^{u,x}\left[ \langle M^{\varphi}\rangle_\tau\right]+C(r)\nonumber\\
 &\le E^{u,x}\left[\Vert\varphi\Vert^2_{L^\infty(\Omega)}\tau\right]+C(r)
 \label{eq:L2_bounded_expMGs_minus_1_over_delta_n}.
\end{align}
Thus the sequence $(\delta_n^{-1}(\mathcal{E}(M^{\delta_n\varphi})_\tau-1))_n$ is bounded in $L^2(\mu^{u,x})$. Since $M^\varphi_\tau$ belongs to $L^2(\mu^{u,x})$ and since $L^2(\mu^{u,x})$ is a vector space, it follows that the sequence $(\delta_n^{-1}(\mathcal{E}(M^{\delta_n\varphi})_\tau-1)-M^\varphi_\tau)_n$ is bounded in $L^2(\mu^{u,x})$. We can then apply Lebesgue's dominated convergence theorem and the identity (\ref{eq:D_minus_1_over_delta_minus_M}) to obtain
\begin{equation*}
 \lim_{n\to\infty}E^{u,x}\left[\left(\frac{\mathcal{E}(M^{\delta_n\varphi})_\tau-1}{\delta_n}-M^{\varphi}_\tau\right)^2\right]=E^{u,x}\left[\lim_{n\to\infty}\left(\frac{\mathcal{E}(M^{\delta_n\varphi})_\tau-1}{\delta_n}-M^{\varphi}_\tau\right)^2\right]=0,
\end{equation*}
as desired.
\end{proof}

\subsection{First variation of control functional}
\label{ssec:first_variation}
In order to derive an expression for the first variation of the control functional $\phi$ defined in (\ref{eq:optimisation_inequality}), we shall proceed by showing the convergence in $L^1(\mu^{u,x})$ of a sequence of finite differences of $\phi^{\sigma,x}$. For a given sequence $(\delta_n)_n$ and function $\varphi$ that satisfy the conditions of Lemma \ref{lem:dominated_convergence_theorem_application}, we define the finite difference random variables by
\begin{equation}\label{eq:h_n_firstvariation}
 h_n^{\sigma,u,\varphi}:= K^{\sigma,u,0}\frac{\mathcal{E}(M^{\delta_n\varphi})_\tau-1}{\delta_n}+\frac{1}{\sigma}\left(\langle M^u,M^\varphi\rangle_\tau+\frac{\delta_n}{ 2} \langle M^\varphi\rangle_\tau\right)\mathcal{E}(M^{\delta_n\varphi})_\tau.
\end{equation}
We show that the $(h_n^{\sigma,u,\varphi})_n$ are indeed the appropriate finite differences in the next
\begin{mylem}\label{lem:barphi_at_c_plus_varphi_minus_barphi_at_c}
 Suppose that Assumption \ref{asmp:regularity_asmp_dV_domain} holds, that $u\in C^\alpha(\overline{\Omega})$, $\varphi\in L^\infty(\Omega)$ and $(\delta_n)_n$ be an arbitrary nonzero sequence satisfying $\delta_n\to 0$ as $n\to\infty$. Let $(h_n^{\sigma,u,\varphi})_n$ be the corresponding sequence given by (\ref{eq:h_n_firstvariation}). Then $(h_n^{\sigma,u,\varphi})_n\subset L^1(\mu^{u,x})$ and
\begin{equation*}
 \frac{\phi^{\sigma,x}(u+\delta_n\varphi)-\phi^{\sigma,x}(u)}{\delta_n}=E^{u,x}\left[h_n^{\sigma,u,\varphi}\right]\hskip2ex \forall n\in\mathbb{N}.
\end{equation*} 
\end{mylem}\begin{proof}
In order to show that $\left\{h^{\sigma,u,\varphi}_n\vert n\in\mathbb{N}\right\}\subset L^1(\mu^{u,x})$, we shall apply the Cauchy-Schwarz inequality individually to the three summands on the right-hand side of (\ref{eq:h_n_firstvariation}). Recall the inequality of arithmetic and geometric means,
\begin{equation}\label{eq:ineq_arith_geom_means}
 \vert ab\vert\le 2^{-1}(a^2+b^2),
\end{equation}
with equality holding if and only if $a=b$, and the Kunita-Watanabe inequality
\begin{equation}\label{eq:kunita_watanabe_inequality}
 \langle M^{\varphi},M^\zeta\rangle_\tau\le \left(\langle M^\varphi\rangle_\tau\langle M^\zeta\rangle_\tau\right)^{1/2}
\end{equation}
that holds almost surely, with equality holding almost surely if and only if $\varphi=\zeta$ except on a set of Lebesgue measure zero. The stated assumptions guarantee that $W\in L^2(\mu^{u,x})$, by Lemma \ref{lem:even_moments_of_W_wrt_mu_u_x_are_finite}; that the sequences $(\delta_n^{-1}(\mathcal{E}(M^{\delta_n\varphi})\tau-1))_n$ and $(\mathcal{E}(M^{\delta_n\varphi})_\tau)_n$ are bounded sequences in $L^2(\mu^{u,x})$, by (\ref{eq:L2_bounded_expMGs_minus_1_over_delta_n}) and (\ref{eq:L2_bounded_expMGs}); and that integer moments of quadratic variation processes belong to $L^2(\mu^{u,x})$, by Corollary \ref{cor:bounded_perturbations_yield_Lq_bounded_exp_MGs} (ii). Since $K^{\sigma,u,0}\in L^2(\mu^{u,x})$ by Corollary \ref{cor:nonequilibrium_estimator_moments_of_all_orders} and 
 \begin{equation}\label{eq:inequality_02}
 E^{u,x}\left[\vert \langle M^u,M^\varphi\rangle_\tau\vert^2\right]\le E^{u,x}\left[ \langle M^u\rangle_\tau+\langle M^\varphi\rangle_\tau^2\right]
\end{equation}
by (\ref{eq:ineq_arith_geom_means}) and (\ref{eq:kunita_watanabe_inequality}), it follows that $\left\{h_n\vert n\in\mathbb{N}\right\}\subset L^1(\mu^{u,x})$, thus proving the first statement.

To prove the second statement, observe that by Lemma \ref{lem:dominated_convergence_theorem_application}, it follows that each $\mathcal{E}(M^{\delta_n\varphi})$ is a $L^2(\mu^{u,x})$-bounded exponential martingale. Since $L^q$-boundedness for $q>1$ implies uniform integrability, we can apply the reweighting formula (\ref{eq:reweighting_formula}) to express $\phi^{\sigma,x}(u+\delta_n\varphi)$ as an expectation with respect to $\mu^{u,x}$. By the definition of the quadratic variation, we have for any stopping time $\tau$ that
\begin{equation}\label{eq:A_of_c_plus_varphi}
 \langle M^{u+\delta_n\varphi}\rangle_\tau=\langle M^{u}\rangle_\tau+\delta_n^2\langle M^\varphi\rangle_\tau+2\delta_n\langle M^u,M^\varphi\rangle_\tau.
\end{equation}
By the reweighting formula and (\ref{eq:A_of_c_plus_varphi}) we therefore have
\begin{equation*}
 \phi^{\sigma,x}(u+\delta_n\varphi)=E^{u,x}\left[\left(W+\frac{\langle M^u\rangle_\tau}{2\sigma}+\frac{\delta_n^2}{2\sigma}\langle M^\varphi\rangle_\tau+\frac{\delta_n}{\sigma}\langle M^u,M^\varphi\rangle_\tau\right)\mathcal{E}(M^{\delta_n\varphi})_\tau\right].
\end{equation*}
Taking differences yields the desired conclusion.
\end{proof}

We now define the limiting random variable to be
\begin{equation}\label{eq:h_firstvariation}
 h^{\sigma,u,\varphi}:=K^{\sigma,u,0}M^{\varphi}_\tau+\sigma^{-1}\langle M^u,M^\varphi\rangle_\tau,
\end{equation}
and show that the expected value of $h^{\sigma,u,\varphi}$ is equal to the first variation of $\phi^{\sigma,x}(u)$ in the direction of $\varphi$:
\begin{mythm}\label{thm:first_variation}
 If the hypotheses of Lemma \ref{lem:barphi_at_c_plus_varphi_minus_barphi_at_c} are satisfied, then $h\in L^1(\mu^{u,x})$, $h_n$ converges to $h$ in $L^1(\mu^{u,x})$, and the first variation of $\phi^{\sigma,x}(u)$ along $\varphi$ satisfies
 \begin{equation}\label{eq:first_variation}
  \Phi^{\sigma,x}(u;\varphi):=\lim_{n\to\infty}\frac{\phi^{\sigma,x}(u+\delta_n\varphi)-\phi^{\sigma,x}(u)}{\delta_n}=E^{u,x}[h^{\sigma,u,\varphi}].
 \end{equation}
\end{mythm}\begin{proof}
The third conclusion follows from the second, since 
 \begin{equation}\label{eq:abs_expectation_less_than_expectation_abs}
  \left\vert E^{u,x}\left[ h^{\sigma,u,\varphi}-h_n^{\sigma,u,\varphi}\right]\right\vert\le E^{u,x}\left[ \left\vert h^{\sigma,u,\varphi}-h_n^{\sigma,u,\varphi}\right\vert\right].
 \end{equation}
 The first conclusion, that $h^{\sigma,u,\varphi}\in L^1(\mu^{u,x})$, follows from the observation that 
  \begin{align*}
  E^{u,x}\left[\vert h^{\sigma,u,\varphi}\vert\right]&\le E^{u,x}\left[ \vert K^{\sigma,u,0} M^\varphi_\tau\vert\right] +\sigma^{-1}E^{u,x}\left[\vert \langle M^u,M^\varphi\rangle_\tau\vert \right],
 \end{align*}
and by using the Cauchy-Schwarz inequality, Corollary \ref{cor:nonequilibrium_estimator_moments_of_all_orders}, and (\ref{eq:inequality_02}).
 
It remains to show the convergence in $L^1(\mu^{u,x})$ of $h_n^{\sigma,u,\varphi}$ to $h^{\sigma,u,\varphi}$. We apply the triangle inequality to obtain
\begin{align}
 \vert h_n^{\sigma,u,\varphi}&-h^{\sigma,u,\varphi}\vert\le \left\vert K^{\sigma,u,0}\left(\frac{\mathcal{E}(M^{\delta_n\varphi})_\tau-1}{\delta_n}-M^{\varphi}_\tau\right)\right\vert
 \label{eq:inequality_for_L1_cvgce_first_variation}
\\
&+\frac{1}{\sigma}\left(\left\vert\langle M^u,M^\varphi\rangle_\tau\left(\mathcal{E}(M^{\delta_n\varphi})_\tau-1\right)\right\vert\right)+\frac{\vert\delta_n\vert}{ 2\sigma}\langle M^\varphi\rangle_\tau \mathcal{E}(M^{\delta_n\varphi})_\tau.
\nonumber
\end{align}
Recall that $(\mathcal{E}(M^{\delta_n\varphi})_\tau)_n$ is a bounded sequence in $L^2(\mu^{u,x})$, by (\ref{eq:L2_bounded_expMGs}). Therefore, by the Cauchy-Schwarz inequality and Lebesgue's dominated convergence theorem, the third term on the right-hand side converges to zero in $L^1(\mu^{u,x})$. By the Cauchy-Schwarz inequality, the bound (\ref{eq:inequality_02}), and the convergence result (\ref{eq:convergence_of_expMG_to_1}), the second term on the right-hand side converges to zero in $L^1(\mu^{u,x})$. By the Cauchy-Schwarz inequality, Corollary \ref{cor:nonequilibrium_estimator_moments_of_all_orders}, and the convergence result (\ref{eq:convergence_of_expMG_minus_1_over_delta}), the first term on the right-hand side converges to zero in $L^1(\mu^{u,x})$. This completes the proof.
\end{proof}

\begin{mycor}\label{cor:useful_relation_first_variation}
The first variation $\Phi^{\sigma,x}(u;\varphi)$ satisfies
\begin{equation}\label{eq:useful_relation_first_variation}
\Phi^{\sigma,x}(u;\varphi)=E^{u,x}\left[K^{\sigma,u,1}M^\varphi_\tau\right].
\end{equation}
In particular, the first variation $\Phi^{\sigma,x}(u;\varphi)$ vanishes if and only if $K^{\sigma,u,1}$ and $M^\varphi_\tau$ are uncorrelated with respect to $\mu^{u,x}$.
\end{mycor}\begin{proof}
The first assertion follows from the definition (\ref{eq:ch3_nonequilibrium_estimator}) of $K^{\sigma,u,\alpha}$ and the fact that $E^{u,x}[\langle M^u,M^\varphi\rangle_\tau]=E^{u,x}[M^u_\tau M^\varphi_\tau]$. The second assertion follows from (\ref{eq:useful_relation_first_variation}).
\end{proof}

Corollary \ref{cor:useful_relation_first_variation} is consistent with the assertion made earlier, that the optimal control that solves the stochastic optimal control problem (\ref{eq:stochastic_optimal_control_problem}) yields an optimal importance sampling measure, i.e. a zero-variance estimator for the value function $F$ for any given pair of arguments.

\subsection{First variation of first variation}
\label{ssec:mixed_second_variation}

In this section, we derive an expression for the first variation of the functional $\Phi^{\sigma,x}(u;\varphi)$ along an admissible perturbation $\zeta$, where the argument $u$ is perturbed and the parameters $\sigma$, $x$, and $\varphi$ are held fixed. As in \S\ref{ssec:first_variation}, we shall first define for a given sequence $(\delta_n)_n$ a sequence of finite differences:
\begin{align}
 h^{\sigma,u,\varphi,\zeta}_n:=&\left(K^{\sigma,u,0}M^\varphi_\tau+\frac{1}{\sigma}\langle M^u,M^\varphi\rangle_\tau\right)\frac{\mathcal{E}(M^{\delta_n\zeta})_\tau-1}{\delta_n}
 \label{eq:h_n_mixed_second_variation}
 \\
 &+\frac{1}{\sigma}\left(\langle M^u,M^\zeta\rangle_\tau M^\varphi_\tau+\langle M^\varphi,M^\zeta\rangle_\tau+\frac{\delta_n}{2}\langle M^\zeta\rangle_\tau\right)\mathcal{E}(M^{\delta_n\zeta})_\tau.
 \nonumber
 \end{align}
 The following result is analogous to Lemma \ref{lem:barphi_at_c_plus_varphi_minus_barphi_at_c}:
 \begin{mylem}\label{lem:firstvar_at_cpluspert_minus_firstvar_at_c}
  Suppose that Assumption \ref{asmp:regularity_asmp_dV_domain} holds, that $u\in C^\alpha(\overline{\Omega})$, that $\varphi,\zeta\in L^\infty(\Omega)$, and that $(\delta_n)_n$ is an arbitrary nonzero sequence satisfying $\delta_n\to 0$ as $n\to\infty$. Let $(h^{\sigma,u,\varphi,\zeta})_n$ be the corresponding sequence given by (\ref{eq:h_n_mixed_second_variation}). Then $(h^{\sigma,u,\varphi,\zeta}_n)_ n\subset L^1(\mu^{u,x})$ and 
  \begin{equation*}
 \frac{\Phi^{\sigma,x}(u+\delta_n\zeta;\varphi)-\Phi^{\sigma,x}(u;\varphi)}{\delta_n}=E^{u,x}\left[h^{\sigma,u,\varphi,\zeta}_n\right].
\end{equation*}
 \end{mylem}\begin{proof}
 We first prove that the terms on the first row of (\ref{eq:h_n_mixed_second_variation}) belong to $L^1(\mu^{u,x})$. By the triangle inequality and the inequality (\ref{eq:ineq_arith_geom_means}) of arithmetic and geometric means, it holds that
 \begin{align*}
\left( K^{\sigma,u,0}M^\varphi_\tau+\frac{1}{\sigma}\langle M^u,M^\varphi\rangle_\tau\right)&\frac{\mathcal{E}(M^{\delta_n\zeta})_\tau-1}{\delta_n} 
  \\
  &\le \vert K^{\sigma,u,0}M^\varphi_\tau\vert^2+ 2\left\vert \frac{\mathcal{E}(M^{\delta_n\zeta})_\tau-1}{\delta_n}\right\vert^2+\sigma^{-2}\langle M^u,M^\varphi\rangle_\tau^2.
 \end{align*}
The third term on the right-hand side belongs to $L^1(\mu^{u,x})$ by (\ref{eq:inequality_02}), the second term on the right-hand side belongs to $L^1(\mu^{u,x})$ by (\ref{eq:L2_bounded_expMGs_minus_1_over_delta_n}), and the first term on the right-hand side belongs to $L^1(\mu^{u,x})$, by statement (ii) of Corollary \ref{cor:bounded_perturbations_yield_Lq_bounded_exp_MGs} and \ref{cor:nonequilibrium_estimator_moments_of_all_orders}. 

Of the terms on the second row of (\ref{eq:h_n_mixed_second_variation}), we only need to show that the product $\langle M^u,M^\zeta\rangle_\tau M^\varphi_\tau\in L^2(\mu^{u,x})$, since the other terms were shown to belong to $L^1(\mu^{u,x})$ in the proof of Lemma \ref{lem:barphi_at_c_plus_varphi_minus_barphi_at_c}. By the Cauchy-Schwarz inequality, (\ref{eq:ineq_arith_geom_means}), and (\ref{eq:kunita_watanabe_inequality}),
\begin{equation}\label{eq:ineq_02_mixed_second_variation}
 E^{u,x}\left[\left(\langle M^u,M^\zeta\rangle_\tau M^\varphi_\tau\right)^2\right]\le \left(E^{u,x}\left[\langle M^u\rangle_\tau^4+\langle M^\zeta\rangle_\tau^4\right]E^{u,x}\left[ \langle M^\varphi\rangle_\tau\right]\right)^{1/2},
\end{equation}
which completes the proof that $h^{\sigma,u,\varphi,\zeta}_n\in L^1(\mu^{u,x})$ for all $n$. 

The second conclusion of the lemma follows in the same manner as for the second conclusion of Lemma \ref{lem:barphi_at_c_plus_varphi_minus_barphi_at_c}, i.e. by expanding the quadratic variation according to (\ref{eq:A_of_c_plus_varphi}), using the reweighting formula (\ref{eq:reweighting_formula}), and taking differences.
\end{proof}

We define the limiting random variable of the sequence $(h_n^{\sigma,u,\varphi,\zeta})_n$ by
\begin{equation}
 h^{\sigma,u,\varphi,\zeta}:=K^{\sigma,u,0}M^\varphi_\tau M^\zeta_\tau+\frac{1}{\sigma}\left(\langle M^\varphi,M^\zeta\rangle_\tau+M^\zeta_\tau\langle M^u,M^\varphi\rangle_\tau +M^\varphi_\tau\langle M^u,M^\zeta\rangle_\tau \right),
 \label{eq:h_mixed_second_variation}
\end{equation}
and show that, in expectation, $h^{\sigma,u,\varphi,\zeta}$ equals the first variation of $\Phi^{\sigma,x}(u;\varphi)$ along $\zeta$:
\begin{mythm}\label{thm:first_variation_of_first_variation}
 If the hypotheses of Lemma \ref{lem:firstvar_at_cpluspert_minus_firstvar_at_c} are satisfied, then $h^{\sigma,u,\varphi,\zeta}$ belongs to $L^1(\mu^{u,x})$, $h_n^{\sigma,u,\varphi,\zeta}$ converges to $h^{\sigma,u,\varphi,\zeta}$ in $L^1(\mu^{u,x})$, and the first variation of $\Phi^{\sigma,x}(u;\varphi)$ along $\zeta$ satisfies
 \begin{equation}\label{eq:mixed_second_variation}
  \Phi^{\sigma,x}(u;\varphi,\zeta):=\lim_{n\to\infty} \frac{\Phi^{\sigma,x}(u+\delta_n\zeta;\varphi)-\Phi^{\sigma,x}(u;\varphi)}{\delta_n}=E^{u,x}\left[h^{\sigma,u,\varphi,\zeta}\right].
 \end{equation}
\end{mythm}\begin{proof}
 As in the proof of Theorem \ref{thm:first_variation}, the third conclusion follows from the second, by (\ref{eq:abs_expectation_less_than_expectation_abs}). The first conclusion, that $h^{\sigma,u,\varphi,\zeta}$ belongs to $L^1(\mu^{u,x})$ follows from Corollary \ref{cor:nonequilibrium_estimator_moments_of_all_orders}, (\ref{eq:inequality_02}), and (\ref{eq:ineq_02_mixed_second_variation}).
 
 It remains to prove the second conclusion, i.e. the convergence of $h_n^{\sigma,u,\varphi,\zeta}$ to $h^{\sigma,u,\varphi,\zeta}$ in $L^1(\mu^{u,x})$. We apply the triangle inequality to obtain
\begin{align*}
 \vert h^{\sigma,u,\varphi,\zeta}_n-h^{\sigma,u,\varphi,\zeta}\vert\le & \left( \vert K^{\sigma,u,0} M^\varphi_\tau\vert+\sigma^{-1}\vert \langle M^u,M^\varphi\rangle_\tau\vert\right)\left\vert \frac{\mathcal{E}(M^{\delta_n\zeta})_\tau-1}{\delta_n}-M^\zeta_\tau\right\vert
 \\
 &+
 \frac{1}{\sigma} \left(\left\vert\langle M^\varphi,M^\zeta\rangle_\tau\right\vert+\left\vert\langle M^u,M^\zeta\rangle_\tau M^\varphi_\tau\right\vert\right) \left\vert \mathcal{E}(M^{\delta_n\zeta})_\tau-1\right\vert
 \\
 &+\frac{\vert\delta_n\vert}{2\sigma}\langle M^\zeta\rangle_\tau \mathcal{E}(M^{\delta_n\zeta})_\tau.
\end{align*}
The $L^1(\mu^{u,x})$ convergence to zero of the term on the third row of the right-hand side was shown in the proof of Theorem \ref{thm:first_variation}. Of the two terms on the second row, only the convergence of the product of $\vert \langle M^u,M^\zeta\rangle_\tau M^\varphi_\tau\vert$ and $\vert \mathcal{E}(M^{\delta_n\zeta})_\tau-1\vert$ has not been proven, but this follows by the Cauchy-Schwarz inequality, (\ref{eq:ineq_02_mixed_second_variation}), and (\ref{eq:convergence_of_expMG_to_1}). The $L^1(\mu^{u,x})$ convergence to zero of the terms on the first row follows from the Cauchy-Schwarz inequality and the convergence result (\ref{eq:convergence_of_expMG_minus_1_over_delta}).
\end{proof}

We now use Theorem \ref{thm:first_variation_of_first_variation}, along with some standard results from the calculus of variations concerning second-order conditions for convexity of $C^2$ functionals defined on Banach spaces, in order to specify a sufficient condition for convexity of the control functional $\phi$ defined in (\ref{eq:optimisation_inequality}).
\begin{mythm}\label{thm:uniform_lower_bound_for_W_sufficient_cond_for_convexity_undiscretised_functional}
 Assume that the hypotheses of Theorem \ref{thm:first_variation_of_first_variation} hold. Let $\sigma>0$ and $x\in\Omega\setminus\partial\Omega$ be fixed. If the terminal cost $\kappa_t$ satisfies the lower bound
 \begin{equation}\label{eq:uniform_lower_bound_for_W_sufficient_cond_for_convexity_undiscretised_functional}
  \kappa_t(y)\ge\sigma^{-1}\hskip2ex\forall y\in\partial\Omega, 
 \end{equation}
 then the undiscretised control functional $\phi^{\sigma,x}:C^{1,\alpha}(\overline{\Omega};\mathbb{R}^d)\to \mathbb{R}$ is strictly convex. 
\end{mythm}\begin{proof}
 By setting $\varphi=\zeta$ in (\ref{eq:h_mixed_second_variation}) and taking expectations, we obtain the following expression for the second variation of $\phi^{\sigma,x}$ at $u$ in the direction of $\varphi$:
 \begin{equation}\label{eq:second_variation}
  \Phi^{\sigma,x}(u;\varphi,\varphi)=E^{u,x}\left[K^{\sigma,u,0}(M^\varphi_\tau)^2+\sigma^{-1}\left(\langle M^\varphi\rangle_\tau+2\langle M^u,M^\varphi\rangle_\tau M^\varphi_\tau\right)\right].
 \end{equation}
Recall (see, e.g. \cite[Section 2.9, Corollary 4]{zeidler}) that if the second variation of a $C^2$ functional on a Banach space is positive definite, then the functional is strictly convex. With this in mind, we show that (\ref{eq:uniform_lower_bound_for_W_sufficient_cond_for_convexity_undiscretised_functional}) suffices to guarantee positive definiteness of the second variation. Let $u\in C^{1,\alpha}(\overline{\Omega})$ and $0\ne \varphi\in C^{1,\alpha}(\overline{\Omega})$ be arbitrary. Observe that, by the Kunita-Watanabe inequality and Young's inequality, we have
\begin{equation*}
 \vert\langle M^u,M^\varphi\rangle_\tau M^\varphi_\tau\vert\le \left(\langle M^u\rangle_\tau\langle M^\varphi\rangle_\tau\right)^{1/2}\vert M^\varphi_\tau\vert\le 4^{-1}\langle M^u\rangle_\tau (M^\varphi_\tau)^2+\langle M^\varphi\rangle_\tau.
\end{equation*} 
Therefore, we have
\begin{align}
\Phi^{\sigma,x}(u;\varphi,\varphi)&\ge E^{u,x}\left[K^{\sigma,u,0}(M^\varphi_\tau)^2-\frac{2}{\sigma}\left(\frac{1}{4}\langle M^u\rangle_\tau (M^\varphi_\tau)^2+\langle M^\varphi\rangle_\tau\right)\right]
\nonumber
\\
&=E^{u,x}\left[ (W-\sigma^{-1})(M^\varphi_\tau)^2\right],
\label{eq:inequality_for_positive_definiteness}
\end{align}
where we used the definition (\ref{eq:ch3_nonequilibrium_estimator}) of $K^{\sigma,u,\alpha}$ and the It\^{o} isometry to obtain (\ref{eq:inequality_for_positive_definiteness}). Since $x\in\Omega\setminus\partial\Omega$ and $\kappa_r$ is nonnegative by Assumption \ref{asmp:regularity_asmp_dV_domain}, it follows from (\ref{eq:uniform_lower_bound_for_W_sufficient_cond_for_convexity_undiscretised_functional}) that $W-\sigma^{-1}$ is strictly positive $\mu^{u,x}$-almost surely. Furthermore, since $\varphi\ne 0$ and since $x\in\Omega\setminus\partial\Omega$, it follows that $(M^\varphi_\tau)^2=\langle M^\varphi\rangle_\tau$ is strictly positive $\mu^{u,x}$-almost surely. This proves that the the random variable inside the expectation in (\ref{eq:inequality_for_positive_definiteness}) is strictly positive $\mu^{u,x}$-almost surely, and thus the variation $\Phi^{\sigma,x}(u;\varphi,\varphi)$ of the control functional $\phi^{\sigma,x}$ is positive definite on $C^{1,\alpha}(\overline{\Omega})$.
\end{proof}

\paragraph{Remark on second-order characterisation of convexity}
\label{para:proving_strict_convexity_via_second_order_characterisation}
 The advantage of proving strict convexity of the control functional $\phi^{\sigma,x}$ by proving positive definiteness of the second variation is that that one can bound the second variation from below by a quantity that does not contain any exponential martingales, as we have shown above. This leads to a simple sufficient condition (\ref{eq:uniform_lower_bound_for_W_sufficient_cond_for_convexity_undiscretised_functional}) for positive definiteness. On the other hand, one can verify using the reweighting formula that the zeroth- and first-order characterisations of convexity involve expressions that contain exponential martingale terms, and thus the sufficient conditions obtained from these characterisations are more complicated. Furthermore, for the zeroth-order characterisation, one needs to check that the Radon-Nikodym derivative exists and is a uniformly integrable exponential martingale, in order to express all the expected values that appear in the zeroth-order characterisation in terms of a common path measure $\mu^{u',x}$. Although Novikov's condition is well-known in stochastic analysis for being a sufficient condition for uniform integrability, we do not know of any sufficient conditions for Novikov's condition that can be verified for arbitrary bounded domains of sufficient regularity and stopping times that are integrable but not almost-surely bounded, aside from the hypotheses of Theorem \ref{thm:suffcon_Lq_boundedness_of_stopped_exp_MG}. However, once the hypotheses of Theorem \ref{thm:suffcon_Lq_boundedness_of_stopped_exp_MG} hold, then it follows that the second variation exists, and one can prove the strict convexity of the control functional using the second-order characterisation, as we have done above.

Given that we have imposed strong conditions on the data of the stochastic optimal control problem, such as boundedness and regularity of the domain $\Omega$ and restriction to feedback controls belonging to $C^{1,\alpha}(\overline{\Omega};\mathbb{R}^d)$, it seems reasonable that we obtain strong results on the control functional. Let $\Pi_2(B)$ be the set of all predictable processes that are adapted to the filtration generated by the Brownian motion $B$. Given $u\in C^\alpha(\overline{\Omega};\mathbb{R}^d)$, define the function $\Vert \cdot\Vert_{B,u}:\Pi_2(B)\to [0,\infty)$ according to
\begin{equation}\label{eq:norm_on_predictable_processes}
 \Vert \varphi\Vert_{B,u}:=\left(E^{u,x}\left[\langle M^\varphi\rangle_\tau\right]\right)^{1/2}.
\end{equation}
It follows from the definition of the quadratic variation process that $\Vert\cdot\Vert_{B,u}$ defines a norm on $\Pi_2(B)$. Since feedback processes $u_t:=u(X^u_t)$ are predictable processes, it follows that $\Vert\cdot\Vert_{B,u}$ induces a norm on the set of feedback control processes, including those generated by feedback controls belonging to $C^{1,\alpha}(\overline{\Omega};\mathbb{R}^d)$. We can therefore consider $\Vert\cdot\Vert_{B,u}$ as a norm on $C^{1,\alpha}(\overline{\Omega};\mathbb{R}^d)$. We then have the following
\begin{mycor}\label{cor:strong_convexity_on_predictable_control_processes}
 Suppose that 
 \begin{equation}\label{eq:strong_convexity_parameter}
  \gamma:=\gamma(\sigma,\kappa_t):=\min_{y\in\partial\Omega } \kappa_t(y)-\sigma^{-1}
 \end{equation}
 is strictly positive. Then for all $u\in C^{1,\alpha}(\overline{\Omega};\mathbb{R}^d)$, the second variation $\Phi^{\sigma,x}(u;\varphi,\varphi)$ is a coercive functional of $\varphi$ with parameter $\gamma$ on $C^{1,\alpha}(\overline{\Omega};\mathbb{R}^d)$, in the sense that 
 \begin{equation}\label{eq:strong_positive_definiteness_of_second_variation}
 \Phi^{\sigma,x}(u;\varphi,\varphi)\ge \gamma(\sigma,\kappa_t)\Vert \varphi\Vert_{B,u}^2
 \end{equation}
 for all $0\ne \varphi\in C^{1,\alpha}(\overline{\Omega};\mathbb{R}^d)$.
\end{mycor}

We emphasise that Theorem \ref{thm:uniform_lower_bound_for_W_sufficient_cond_for_convexity_undiscretised_functional} are consistent with the classical result (Theorem \ref{thm:feynman_kac_representation_uniformly_elliptic_case_uniformly_holder}) concerning the existence of a unique solution to the linear Dirichlet problem. In the context of linear elliptic Dirichlet boundary value problems, $\kappa_t$ is a continuous function on a compact set. Therefore, one can always ensure that the sufficient condition (\ref{eq:uniform_lower_bound_for_W_sufficient_cond_for_convexity_undiscretised_functional}) holds, by considering the translated terminal cost function $\widehat{\kappa_t}:=\kappa_t+C$, where $C=\max\left\{0,-\gamma(\sigma,\kappa_t)\right\}$.

\section{Approximation by a finite-dimensional subspace}
\label{sec:approximation_by_fin_dim_subspace}
In this section, we consider some consequences of the strict convexity result Theorem \ref{thm:uniform_lower_bound_for_W_sufficient_cond_for_convexity_undiscretised_functional}, given that Assumption \ref{asmp:regularity_asmp_dV_domain} holds. The first consequence of Theorem \ref{thm:uniform_lower_bound_for_W_sufficient_cond_for_convexity_undiscretised_functional} that we shall consider (Theorem \ref{thm:well_posedness_of_gradient_descent} below) concerns the well-posedness of the gradient descent algorithm proposed by Hartmann and Sch\"{u}tte in \cite{hartmannschuette_efficient}. The second consequence concerns the best approximation of the value function $F(\sigma,\cdot)$ that solves the Hamilton-Jacobi-Bellman equation of the stochastic optimal control problem (\ref{eq:stochastic_optimal_control_problem}).

Fix a finite collection $b:=(b_i)_{i=1}^n$ of linearly independent, nonconstant functions $b_i\in C^{2,\alpha}(\overline{\Omega};\mathbb{R})$, and let $S_i:=\left\{x\in\overline{\Omega}\ \vert\ b_i(x)\ne 0\right\}$ denote the support of $b_i$ in $\overline{\Omega}$. We shall assume that 
\begin{equation}\label{eq:supports}
 \lambda(S_i)>0,\hskip1ex\cup_{i=1}^n\overline{S_i}=\overline{\Omega},
\end{equation}
i.e. the support of each function $b_i$ has strictly positive Lebesgue measure, and the closures of the supports cover $\overline{\Omega}$. Let $\varphi_i:=\nabla b_i$ and consider the following subspace of $C^{1,\alpha}(\overline{\Omega};\mathbb{R}^d)$:
\begin{equation}\label{eq:approximating_subspace}
 \mathcal{U}_{\text{Markov}}(b):=\left\{u\in\mathcal{U}_{\text{Markov}}\ \biggr\vert\ u=u^a:=\sum_i a_i\nabla b_i\text{ for some }a\in\mathbb{R}^n\right\}.
\end{equation}
By linear independence of the basis vector fields $(\varphi_i)_i$, it holds that every $u\in\mathcal{U}_{\text{Markov}}(b)$ admits a unique representation $u=u^a$. Therefore, the restriction of the control functional $\phi^{\sigma,x}$ to the subspace $\mathcal{U}_{\text{Markov}}(b)$ may be uniquely represented by
\begin{equation}\label{eq:restricted_control_functional}
\widehat{\phi}^{\sigma,x}:\mathbb{R}^n\to \mathbb{R},\hskip3ex \widehat{\phi}^{\sigma,x}(a):=\phi^{\sigma,x}(u^a).
\end{equation}
By Corollary \ref{cor:useful_relation_first_variation}, we obtain an expression for the partial derivative of the function $\widehat{\phi}^{\sigma,x}$ with respect to the $i$-th coordinate $a_i$. 

\subsection{An initial value problem}
\label{ssec:initial_value_problem}

Consider the initial value problem
\begin{equation}\label{eq:gradient_descent_ivp}
 \frac{d}{dt}a(s)=-\nabla_a \phi^{\sigma,x}(a(s)),\hskip1ex a(0)\in\mathbb{R}^n.
\end{equation}
The main result of this section concerns the existence of solutions to (\ref{eq:gradient_descent_ivp}). To prove the result, we make the following 
\begin{myasmp}\label{asmp:ergodicity_of_uncontrolled_paths}
  Let $\nabla V\in C^{\alpha}(\overline{\Omega})$ and $\tau$ be defined as in (\ref{eq:first_exit_time}). Then for all $1\le i\le n$, there exists a pair $0\le t_0(i)<t_1(i)\le \tau$ such that for all $t_0(i)<t<t_1(i)$, $X^0_t\in S_i$ with positive $P^x$-probability.
\end{myasmp}
\begin{myprop}\label{prop:ergodicity_of_controlled_paths}
 Suppose that Assumptions \ref{asmp:regularity_asmp_dV_domain} and \ref{asmp:ergodicity_of_uncontrolled_paths} hold, and let $\tau$ be defined as in (\ref{eq:first_exit_time}). If $u^a\in\mathcal{U}_{\text{Markov}}(b)$, then for all $1\le i\le n$, there exists a pair $0\le t_0(i)<t_1(i)\le \tau$ such that for all $t_0(i)<t<t_1(i)$, $X^{u^a}_t\in S_i$ with positive $P^x$-probability.
\end{myprop}
\begin{proof}
 By Assumption \ref{asmp:regularity_asmp_dV_domain}, we may apply Corollary \ref{cor:bounded_perturbations_yield_Lq_bounded_exp_MGs} to assert that $\mathcal{E}((M^{u^a})^\tau)$ is a $L^2(\mu^{0,x})$-bounded martingale, and hence is uniformly integrable with respect to $\mu^{0,x}$. By the change of measure theorem, the restrictions of $\mu^{0,x}$ and $\mu^{u^a,x}$ to $\mathcal{F}_t$ are mutually absolutely continuous for all $t>0$. The desired conclusion follows from Assumption \ref{asmp:ergodicity_of_uncontrolled_paths}.
\end{proof}
\begin{mycor}\label{cor:QV_of_M_varphi_i_bdd_away_from_zero}
 Given an arbitrary $a(0)\in\mathbb{R}^n$ and compact neighbourhood $\overline{N}$ of $a(0)$, it holds that 
 \begin{equation}\label{eq:QV_of_M_varphi_i_bdd_away_from_zero}
  m^x(\overline{N}):=\min_{a\in \overline{N}}\min_{1\le i\le n}E^{u^a,x}\left[\langle M^{\varphi_i}\rangle_\tau\right]>0.
 \end{equation}
\end{mycor}
\begin{proof}
 By the representation (\ref{eq:approximating_subspace}), the distribution $\mu^{u^a,x}$ depends continuously on $a$. Therefore the map $a\mapsto \min_i E^{u^a,x}[\langle M^{\varphi_i}\rangle_\tau]$ is continuous. By Proposition \ref{prop:ergodicity_of_controlled_paths}, the map is strictly positive for every $a\in\mathbb{R}^n$. The conclusion follows by compactness.
\end{proof}

Now consider the restriction to $\mathcal{U}_{\text{Markov}}(b)$ of the stochastic optimal control problem defined in (\ref{eq:stochastic_optimal_control_problem}),
\begin{equation}\label{eq:restricted_stochastic_optimal_control_problem}
 \min_{a\in\mathbb{R}^n}\ \widehat{\phi}^{\sigma,x}(a)\hskip2ex\text{ subject to (2.1) }.
\end{equation}
The restricted problem differs from the original stochastic optimal control problem (\ref{eq:stochastic_optimal_control_problem}) only in the function space over which the control functional is to be minimised. The gradient descent algorithm proposed in \cite{hartmannschuette_efficient} for solving the restricted problem above consisted of the following steps: Given $\mathcal{U}_{\text{Markov}}(b)$, an initial condition $a(0)$, $N_{\text{iter}},N_{\text{traj}}\in\mathbb{N}$, and a sequence of descent steps $(h_\ell)_{0\le \ell \le N_{\text{iter}}-1}\subset (0,\infty)$, we have:
\begin{itemize}
 \item[] \textbf{for} $\ell=1:N_{\text{iter}}$ \textbf{do}
 \begin{itemize}
 \item Sample $N_{\text{traj}}$ trajectories of (\ref{eq:controlled_sde}) with $u=u^{a(i-1)}$
 \item Approximate $\widehat{\phi}^{\sigma,x}(a(\ell-1))$ and $\nabla_a\widehat{\phi}^{\sigma,x}(a(\ell-1))$ by Monte Carlo
 \item Update the coefficient vector $a(\ell-1)$ to $a(\ell)$ via
 \begin{equation}\label{eq:gradient_descent_algorithm}
  a(\ell)=a(\ell-1)-h_{\ell-1}\nabla_a\widehat{\phi}^{\sigma,x}(a(\ell-1))
 \end{equation}
 \end{itemize}
 \item[] \textbf{end}
\end{itemize}
That is, the gradient descent algorithm (\ref{eq:gradient_descent_algorithm}) is the discrete time, finite-sample size analogue of the initial value problem (\ref{eq:gradient_descent_ivp}). We shall study the well-posedness of the gradient descent algorithm via (\ref{eq:gradient_descent_ivp}). The following result shows that, under Assumptions \ref{asmp:regularity_asmp_dV_domain} and \ref{asmp:ergodicity_of_uncontrolled_paths}, the restricted problem (\ref{eq:restricted_stochastic_optimal_control_problem}) is well-posed.

\begin{mythm}\label{thm:well_posedness_of_gradient_descent}
Suppose that the hypotheses of Theorem \ref{thm:first_variation_of_first_variation} are satisfied. If (\ref{eq:uniform_lower_bound_for_W_sufficient_cond_for_convexity_undiscretised_functional}) holds, then there exists a unique $a_\infty\in\mathbb{R}^n$ such that
\begin{equation}\label{eq:unique_equilibrium}
 0=\nabla_a\widehat{\phi}^{\sigma,x}(a_\infty),
\end{equation}
and for every initial condition $a(0)\in\mathbb{R}^n$, there exists a unique solution $(a(t))_{0\le t< \infty}$ of the initial value problem (\ref{eq:gradient_descent_ivp}) to $a_\infty$. If, in addition to the above hypotheses, it holds that 
\begin{equation}\label{eq:non_overlap_condition_on_supports}
 \overline{S_i}\cap\overline{S_j}=\partial S_i\cap\partial S_j,
\end{equation}
then the unique solution converges exponentially to $a_\infty$, with
\begin{equation}\label{eq:rate_of_convergence_under_strong_convexity}
  \vert a_\infty-a(t)\vert^2\le \vert a(0)-a_\infty\vert^2 \exp(-t \gamma m^x(\overline{N}) ),
 \end{equation}
 where $m^x(\overline{N})$ is defined in (\ref{eq:QV_of_M_varphi_i_bdd_away_from_zero}).
\end{mythm}
\begin{proof}
 Since the hypotheses of Theorem \ref{thm:first_variation_of_first_variation} hold, and since  (\ref{eq:uniform_lower_bound_for_W_sufficient_cond_for_convexity_undiscretised_functional}) holds, it follows that the control functional $\phi^{\sigma,x}$ is strictly convex on $C^{1,\alpha}(\overline{\Omega};\mathbb{R}^d)$, by Theorem \ref{thm:uniform_lower_bound_for_W_sufficient_cond_for_convexity_undiscretised_functional}. This implies that the function $\widehat{\phi}^{\sigma,x}$ is twice continuously differentiable and strictly convex on any bounded, open, nonempty convex subset of $\mathbb{R}^n$, so there exists a unique global minimiser that satisfies (\ref{eq:unique_equilibrium}). By the standard results on existence and uniqueness of solutions to ordinary differential equations, it follows that for any initial condition $a(0)$, there exists a unique solution $(a(t))_{t\ge 0}$ that exists for all time.
 
 By Theorem \ref{thm:first_variation_of_first_variation}, (\ref{eq:h_mixed_second_variation}), and linearity of the representation $u^a$, it holds that, given the Hessian $\nabla^2_a\widehat{\phi}^{\sigma,x}:\mathbb{R}^n\to \mathbb{R}^{n\times n}$, arbitrary $a'\in\mathbb{R}^n$ and $0\ne z\in\mathbb{R}^n$,
 \begin{align*}
  z\cdot\nabla^2_a\widehat{\phi}^{\sigma,x}(a')z&=E^{u^{a'},x}\left[K^{\sigma,u^{a'},0}(M^{u^z}_\tau)^2+\sigma^{-1}\left(\langle M^{u^z}\rangle_\tau+2\langle M^{u^{a'}},M^{u^z}\rangle_\tau M^{u^z}_\tau\right)\right].
  \end{align*}
 By (\ref{eq:second_variation}), the value of $z\cdot \nabla^2_a\widehat{\phi}^{\sigma,x}(a')z$ equals the value of the second variation of $\phi^{\sigma,x}$ at $u^{a'}$ along the element $u^z$. By Corollary \ref{cor:strong_convexity_on_predictable_control_processes},
 \begin{equation}\label{eq:coercive_01}
  z\cdot\nabla^2_a\widehat{\phi}^{\sigma,x}(a')z\ge \gamma E^{u^{a'},x}\left[\langle M^{u^z}\rangle_\tau\right].
 \end{equation}
By the chain rule, Taylor's theorem and the property (\ref{eq:unique_equilibrium}), there is a $z\in\mathbb{R}^n$ that is a convex combination of $a(s)$ and $a_\infty$ such that
 \begin{align}
  \frac{d}{dt}\frac{1}{2}\vert a(s)-a_\infty\vert^2 &\le -\frac{1}{2}(a_\infty-a(s))\cdot\nabla^2_a\widehat{\phi}^{\sigma,x}(z) (a_\infty-a(s))
  \nonumber
  \\
  &\le -\frac{1}{2}\gamma E^{u^z,x}\left[\langle M^{u^{a_\infty-a(s)}}\rangle_\tau\right].
 \label{eq:lyapunov_decay_estimate}
 \end{align}
 Observe 
 that by Assumption \ref{asmp:regularity_asmp_dV_domain}, $\mu^{u^a,x}$ is unique, and therefore $X^{u^a}$ has the strong Markov property. Under the condition (\ref{eq:non_overlap_condition_on_supports}) on the supports, it holds that $\varphi_i(x)\cdot\varphi_j(x)$ is zero except on a set of zero Lebesgue measure, and therefore $\langle M^{\varphi_i},M^{\varphi_j}\rangle_\tau$ has vanishing $\mu^{u^a,x}$-expectation. Let $\overline{N}$ be the smallest compact set that contains all convex combinations of $a(t)$ and $a_\infty$ for all $t\ge 0$. It follows from the vanishing of the covariance processes that $\phi^{\sigma,x}$ is strongly convex on $N$, since
 \begin{equation*}
  E^{u^z,x}\left[\langle M^{u^{a_\infty-a(s)}}\rangle_\tau\right]\ge m^x(\overline{N}) \vert a_\infty-a(s)\vert^2.
 \end{equation*}
Substituting the above inequality in (\ref{eq:lyapunov_decay_estimate}), integrating the resulting differential inequality, and applying Gronwall's lemma, we obtain (\ref{eq:rate_of_convergence_under_strong_convexity}).
\end{proof}

\subsection{Approximation of solution to the HJB equation}
\label{ssec:approximation_of_value_function}
In this section, we relate the unique minimiser of the restricted control functional $\widehat{\phi}^{\sigma,x}$ to the best approximation in the finite-dimensional subspace $\mathcal{U}_{\text{Markov}}(b)$ of $F(\sigma,\cdot)$. Recall that, by Theorem \ref{thm:feynman_kac_representation_uniformly_elliptic_case_uniformly_holder}, the function $\psi(\sigma,\cdot)$ for $\psi$ defined in (\ref{eq:mgf}) belongs to the space $C^{2,\alpha}(\overline{\Omega})$. Since $\psi(\sigma,\cdot)=\exp(-\sigma F(\sigma,\cdot))$, this implies that $F(\sigma,\cdot)\in C^{2,\alpha}(\overline{\Omega})$, and since $\overline{\Omega}$ is compact, this implies that $F(\sigma,\cdot)\in L^2(\Omega;\mathbb{R})$. In fact, by the same reasoning, $F(\sigma,\cdot)$ is an element of the Sobolev space $H^{2,p}(\Omega)$ for $1\le p<\infty$. The unique equilibrium $a_\infty$ defined in (\ref{eq:unique_equilibrium}) gives the best approximation in the space $\mathcal{U}_{\text{Markov}}(b)$ of the optimal control $u^{\sigma}_{\text{opt}}$. Therefore, the unique equilibrium $a_\infty$ gives the best approximation of $\nabla_x F(\sigma,\cdot)$, up to scaling. Since we have assumed $\Omega$ to be a connected, bounded domain with $C^{2,\alpha}$-boundary, we may apply the Poincar\'{e} inequality as follows: define
\begin{equation*}
 F_{\text{approx}}(\sigma,x'):=\sum_i (a_\infty)_i b_i(x')+\left(\widehat{\phi}^{\sigma,x}(a_\infty)-\sum_i (a_\infty)_i b_i(x)\right),
\end{equation*}
where the additive constant in the parentheses on the right-hand side ensures that $F_\text{approx}(\sigma,x)=\widehat{\phi}^{\sigma,x}(a_\infty)$. Define $F_{\text{err}}$ and the average value $(F_\text{err})_\Omega$ by
\begin{equation*}
 F_{\text{err}}(\sigma,x):=F(\sigma,x)-F_{\text{approx}}(\sigma,x),\hskip1ex (F_\text{err})_\Omega:=\frac{1}{\text{vol}(\Omega)}\int_\Omega F_{\text{err}}(\sigma,x)dx.
\end{equation*}
Then since $F_{\text{err}}(\sigma,\cdot)\in W^{1,p}(\Omega)$, the Poincar\'{e} inequality yields the estimate
\begin{equation}\label{eq:error_estimate}
 \Vert F_{\text{err}}(\sigma,\cdot)-(F_{\text{err}})_\Omega\Vert_{L^p(\Omega)}\le C(n,p,\Omega)\Vert \nabla_x F_{\text{err}}(\sigma,\cdot)\Vert_{L^p(\Omega)}
\end{equation}
for $1\le p<\infty$ and $C(n,p,\Omega)$ independent of $F_{\text{err}}(\sigma,\cdot)$. The estimate emphasises the importance of the approximation quality of the subspace $\mathcal{U}_{\text{Markov}}(b)$: as the approximation quality of $\mathcal{U}_{\text{Markov}}(b)$ improves, the right-hand side of (\ref{eq:error_estimate}) decreases, giving smaller $L^p$-errors in the approximation of $F(\sigma,\cdot)$.

\section{Conclusion}
\label{sec:conclusion}

In this article, we proved that the control functional of a stochastic optimal control problem is strictly convex, under the conditions that guarantee the existence and uniqueness of a value function that is uniformly H\"{o}lder continuous on the closure of a bounded, $C^{2,\alpha}$ domain. Our results relied on applying Fredholm theory in order to show that the moment generating function of the first exit time with respect to the law $\mu^{u,x}$ for feedback controls $u\in L^\infty(\Omega)$ is finite on any bounded open interval containing the origin. From this result, we obtained the $L^2$-boundedness of the exponential martingale associated to the feedback control. We then used $L^2$-boundedness to derive an expression for the second variation of the control functional. Positive definiteness then followed from a specific lower bound on the terminal cost function. The result of strict convexity holds independently of the dimension $d$ of the domain. The significance of this result is that one may solve the Hamilton-Jacobi-Bellman boundary value problem, and thereby the rare event importance sampling problem, by using methods for convex optimisation in conjunction with Monte Carlo methods. The advantage of doing so is that both classes of methods scale well with dimension.

\section*{Acknowledgments}

The author thanks Christof Sch\"{u}tte, Carsten Hartmann, and Tim Sullivan from the Freie Universit\"{a}t Berlin for their encouragement and support of the thesis on which this article is based, Patrick J. Fitzsimmons from the University of California at San Diego for pointing out Kazamaki's book \cite{kazamaki}, and Giacomo di Ges\`{u} from CERMICS - \'{E}cole des Ponts ParisTech for pointing out an important subtlety concerning the mutual absolute continuity of path measures.

\bibliographystyle{siam}
\bibliography{thesis_abridged}

\end{document}